\documentclass[12pt,leqno]{article}
\usepackage{amsmath, amsthm, amsfonts, amssymb, color}
\usepackage{mathrsfs}
\usepackage[notcite,notref]{showkeys}

\newtheorem{thm}{Theorem}[section]
\newtheorem{cor}[thm]{Corollary}
\newtheorem{lem}[thm]{Lemma}

\newtheorem{rem}[thm]{Remark}

\theoremstyle{definition}

\newcommand{\scr}[1]{\mathscr #1}
\definecolor{wco}{rgb}{0.5,0.2,0.3}

\numberwithin{equation}{section}

\newcommand{\ua}{\uparrow}

\title{{\bf A study of backward stochastic differential equation on a Riemannian manifold}}

\author{
{\bf    Xin Chen$^{a)}$, Wenjie Ye$^{a)}$} 
\thanks{E-mail address: chenxin217@sjtu.edu.cn (X. Chen), yewenjie@sjtu.edu.cn(W.J. Ye)}
\\
\footnotesize {$^{a)}$ School of Mathematical Sciences, Shanghai Jiaotong University, Shanghai 200240, China}
}

\begin{document}

\allowdisplaybreaks

\def\R{\mathbb R} \def\Z{\mathbb Z} \def\ff{\frac} \def\ss{\sqrt}
\def\dd{\delta} \def\DD{\Delta} \def\vv{\varepsilon} \def\rr{\rho}
\def\<{\langle} \def\>{\rangle} \def\GG{\Gamma} \def\gg{\gamma}
\def\ll{\lambda} \def\LL{\Lambda} \def\nn{\nabla} \def\pp{\partial}
\def\d{\text{\rm{d}}} \def\loc{\text{\rm{loc}}} \def\bb{\beta} \def\aa{\alpha} \def\D{\mathbb D}
\def\E{\mathbb E}
\def\p{\mathbf P}
\def\a{\alpha}
\def\O{\Omega}
\newcommand{\Ex}{{\bf E}}
\def\si{\sigma} \def\ess{\text{\rm{ess}}}
\def\beg{\begin} \def\beq{\beg}  \def\F{\scr F}
\def\Ric{\text{\rm{Ric}}} \def\Hess{\text{\rm{Hess}}}\def\B{\scr B}
\def\e{\varepsilon} \def\ua{\underline a} \def\OO{\Omega} \def\b{\mathbf b}
\def\oo{\omega}     \def\tt{\tilde} \def\Ric{\text{\rm{Ric}}}
\def\cut{\text{\rm{cut}}}
\def\P{\mathbb P}
\def\ifn{I_n(f^{\bigotimes n})}
\def\fff{f(x_1)\dots f(x_n)} \def\ifm{I_m(g^{\bigotimes m})} \def\ee{\varepsilon}
\def\C{\text{curl}}
\def\B{\scr B}
\def\S{\scr S}
\def\M{\scr M}
\def\ll{\lambda}
\def\X{\scr X}
\def\T{\mathbb T}
\def\A{\scr A}
\def\LL{\scr L}
\def\gap{\mathbf{gap}}
\def\div{\text{\rm div}}
\def\dist{\text{\rm dist}}
\def\cut{\text{\rm cut}}
\def\supp{\text{\rm supp}}
\def\Var{\text{\rm Var}}
\def\p{\mathbf{p}}
\def\Cov{\text{\rm Cov}}
\def\Cut{\text{\rm Cut}}
\def\le{\leqslant}
\def\ge{\geqslant}
\def\I{\scr I}
\def\coth{\text{\rm coth}}
\def\Dom{\text{\rm Dom}}
\def\Cap{\text{\rm Cap}}
\def\Ent{\text{\rm Ent}}
\def\sect{\text{\rm sect}}\def\H{\mathbb H}
\def\g{\tilde {g}}
\def\o{\omega}
\def\u{\tilde{u}}
\def\c{\tilde{\theta}}\def\w{\tilde{\omega}}
\def\om{\tilde{\Omega}}\def\v{\varepsilon}
\def\U{\tilde{U}}
\def\b{\tilde \beta}
\def\S{\scr S}
\def\D{\mathbb D}
\def\F{\scr F}
\def\L{\Lambda}

\newtheorem{theorem}{Theorem}[section]
\newtheorem{lemma}[theorem]{Lemma}
\newtheorem{proposition}[theorem]{Proposition}
\newtheorem{corollary}[theorem]{Corollary}

\theoremstyle{definition}
\newtheorem{definition}[theorem]{Definition}
\newtheorem{example}[theorem]{Example}
\newtheorem{remark}[theorem]{Remark}
\newtheorem{assumption}[theorem]{Assumption}
\newtheorem*{ack}{Acknowledgement}
\renewcommand{\theequation}{\thesection.\arabic{equation}}
\numberwithin{equation}{section}

\maketitle

\rm

\vskip 10mm

\begin{abstract}
Suppose $N$ is a compact Riemannian manifold, in this paper we will introduce the definition
of $N$-valued BSDE and $L^2(\T^m;N)$-valued BSDE for which the solution are not necessarily staying
in only one local coordinate. Moreover, the global existence of a solution to $L^2(\T^m;N)$-valued BSDE will be proved
without any convexity condition on $N$.
\end{abstract}

\vskip 3mm

\section{Introduction}
Consider the following systems of backward stochastic differential equation (which will be written as BSDE for
simplicity through this paper) in $\R^n$,
\begin{equation}\label{e1-1}
Y_t=\xi-\int_t^T Z_s dB_s-\int_t^T f(s,Y_s,Z_s)ds,\ t\in [0,T].
\end{equation}
Here $\{B_s\}_{s\ge 0}$ is a standard $m$-dimensional Brownian motion defined on a probability space
$(\Omega,\mathscr{F},\P)$, $\xi$ is a $\mathscr{F}_T$-measurable $\R^n$-valued random variable, $\{Y_s\}_{s\in [0,T]}$, $\{Z_s\}_{s\in [0,T]}$ are
$\R^n$-valued predictable process and $\R^{mn}$-valued predictable process respectively. We usually call the function
$f:\Omega\times [0,T]\times \R^n\times \R^{m n}\to \R^n$ the generator of BSDE \eqref{e1-1}.

Bismut \cite{Bi} first introduced the linear version of BSDE \eqref{e1-1}. A breakthrough was
made by Pardoux and Peng \cite{PP1} where the existence of a unique solution to
\eqref{e1-1} was proved under global Lipschitz continuity of generator $f$. Still under the global Lipschitz continuity of $f$,
Pardoux and Peng \cite{PP2} has established an equivalent relation between the systems of
forward-backward stochastic differential equation (which will be written as FBSDE through this paper) and
the solution of a quasi-linear parabolic system. Another important observation by
\cite{PP1,PP2} was that BSDE \eqref{e1-1} could be viewed as a non-linear perturbation of martingale representation
theorem or Feynman-Kac formula.

It is natural to ask what is the variant for BSDE \eqref{e1-1} on a smooth manifold $N$.
When an $n$-dimensional manifold $N$ was endowed with only one local coordinate, Darling \cite{Da1} introduced a kind of
$N$-valued BSDE as follows
\begin{equation}\label{e1-2}
Y_t^k=\xi^k-\sum_{l=1}^m\int_t^T Z_s^{k,l} dB_s^l+\frac{1}{2}
\sum_{l=1}^m\int_t^T \sum_{i,j=1}^n \Gamma_{ij}^k (Y_s)Z_s^{i,l}Z_s^{j,l} ds.
\end{equation}
Here $Y_t=(Y_t^1,\cdots,Y_t^n)$ denotes the components of $Y_t$ under (the only one) local coordinate, and
$\{\Gamma_{ij}^k\}_{i,j,k=1}^n$ are the Christoffel symbols for a fixed affine connection $\Gamma$ on $N$.
The most important motivation to define \eqref{e1-2} is to construct a $\Gamma$-martingale with fixed terminal value
(we refer readers to \cite{E} or \cite{H} for the definition of $\Gamma$-martingale).
 In fact, with
the special choice of generator in \eqref{e1-2}
(which depends on the connection $\Gamma$), the solution $\{Y_t\}_{t\in [0,T]}$ of \eqref{e1-2} is a $\Gamma$-martingale
on $N$ with terminal value $\xi$. Moreover, Blache \cite{B1,B2} investigated a more general $N$-valued BSDE as follows when
$\{Y_t\}$ was restricted in only one local coordinate of $N$,
\begin{equation}\label{e1-3}
Y_t^k=\xi^k-\sum_{l=1}^m\int_t^T Z_s^{k,l} dB_s^l+\frac{1}{2}
\sum_{l=1}^m\int_t^T \sum_{i,j=1}^n \Gamma_{ij}^k (Y_s)Z_s^{i,l}Z_s^{j,l}ds+
\int_t^T f^k(Y_s,Z_s)ds,
\end{equation}
where $f:\Omega \times N\times T^m N \to TN$ is uniformly Lipschitz continuous. Moreover, the Lie group valued BSDE 
has been studied by Estrade and Pontier \cite{EP}, Chen and Cruzeiro \cite{CC}.

On the other hand, the $N$-valued FBSDE is highly related to the heat flow of harmonic map with
target manifold $N$. Partly using some idea of $N$-valued FBSDE, Thalmaier \cite{T} studied several problems
concerning about the singularity for heat flow of harmonic map by probabilistic methods. We also refer readers to
\cite{B1,B2,Da2,GPT,Ke1,Ke2,P,TW} for various methods and applications for the subjects
on $\Gamma$-martingale theory and its connection to the study of heat flow of harmonic map.

For the problems on $N$-valued BSDE mentioned above, there are  two mainly difficulties.
One is the  quadratic growth (for the variable associated with $Z$) term in the generator of \eqref{e1-2} and \eqref{e1-3}, for which the arguments in \cite{PP1,PP2} may not be applied directly to prove global (in time) or local
existence of a solution to \eqref{e1-2} and \eqref{e1-3}.
Kobylanski first proved the global  existence of a unique solution to the scalar valued (i.e. $n=1$)
BSDE \eqref{e1-1} with generator having quadratic growth  and bounded terminal value.
Briand and Hu \cite{BH1,BH2} extended these results to the case where the terminal value may be unbounded.
The problem for multi-dimensional BSDE is more complicated,
Darling \cite{Da1} introduced a kind of condition on the existence of
some doubly convex function, under which the global existence of
a unique solution to \eqref{e1-2} or \eqref{e1-3} has been obtained in \cite{Da1,B1,B2}.
Xing and Zitkovi\'c \cite{XZ} proved global existence of a unique Markovian solution of
\eqref{e1-1} based on the existence of a single convex function.
We also
refer readers to \cite{HT,HR,KLT,Te} and reference therein for various results concerning about the local existence of
a solution to \eqref{e1-1} in $\R^n$ with generator having quadratic growth  under different conditions,
including the boundness for Malliavin derivatives of terminal value(see Kupper, Luo and Tangpi \cite{KLT}),
small $L^\infty$ norm of terminal value
(see Harter and Richou \cite{HR} or Tevzadze \cite{Te}) and the special
diagonal structure of generators(see Hu and Tang \cite{HT}).

Another difficulty for $N$-valued BSDE is the lack of a linear structure for a general manifold $N$. In fact, the expression
\eqref{e1-2} and \eqref{e1-3} only make sense in a local coordinate which is diffeomorphic to an open set of $\R^n$. If
we want to extend \eqref{e1-2} and \eqref{e1-3} to the whole manifold $N$, the multiplication or additive operators
(therefore the It\^o integral term) may not be well defined because of the lack of a linear structure on $N$.
Due to this reason, \cite{Da1,B1,B2}  gave the definition of an $N$-valued BSDE which was restricted in only one
local coordinate. Meanwhile in \cite{CC} and \cite{EP}, the left (or right) translation on a Lie group has been applied 
to provide a linear structure for associated BSDE. 

By our knowledge, for a general $N$, how to define an
$N$-valued BSDE which are not necessarily staying in only one local coordinate is still unknown.
In this paper, we will solve this problem for the case that $N$ is a compact Riemannian manifold, see Definition \ref{d2-1} and \ref{d2-2} in Section \ref{section2}
below. Moreover, as explained above, the existence of a doubly convex or a single convex function is required to
prove the global existence of BSDE whose generator has quadratic growth. The existence
of these convex functions could be verified locally in $N$ (in fact, at every small enough neighborhood), see e.g.
\cite{B1,B2,Ke1}. But except for some special examples (such as Cartan-Hadamard manifold),
it is usually difficult to check whether such a convex function exists globally or not
in $N$. In this paper,  we will also prove the global existence
of an $N$-valued solution to some special BSDE without any convexity conditions mentioned above, see
Theorem \ref{t2-3} and \ref{t2-1} in Section \ref{section2.1}.

 We also give some remarks on our results as follows
\begin{itemize}
\item [(1)] Given a Riemannian metric on $N$, in Definition \ref{d2-1} and \ref{d2-2} we view $N$ as a sub-manifold of ambient
space $\R^L$, so the linear structure on $\R^L$ could be applied in BSDE \eqref{d2-1-1} and \eqref{d2-2-2}.
The key ingredient in \eqref{d2-1-1} and \eqref{d2-2-2} is that the term with quadratic growth is related to second
fundamental form $A$. As illustrated in the proof of Theorem \ref{t2-2}, it will ensure solution of
$\R^L$-valued BSDE \eqref{d2-1-1} to stay in $N$.  The advantage of our definition is that it does not
require the solution to be restricted in only one local coordinate as in \cite{Da1,B1,B2}, therefore
we do not need any extra condition on the generator $f$ in \eqref{d2-1-1} (see e.g. condition $(H)$ in \cite{B1,B2}).
Moreover, as explained in Remark \ref{r2-2}, our definition will be the same as that in \cite{B1,B2} when
we assume that the solution of \eqref{d2-1-1} is situated in only one local coordinate.

\item [(2)] The equation \eqref{d2-2-2} could be viewed as an $N$-valued FBSDE with forward equation
being $x+B_t$ in $\T^m$. In Definition \ref{d2-2}, we study the FBSDE with a.e. initial point $x\in \T^m$.
This kind of solution has been introduced in \cite{BM,MX,ZZ2,ZZ3} to investigate the connection between
FBSDE and weak solution of a quasi-linear parabolic system. The motivation of Definition
\ref{d2-2} is to study the global existence of a solution to $N$-valued BSDE for more general $N$,
especially for that without any convexity condition. Theorem
\ref{t2-1} ensures us to find a global solution of \eqref{d2-2-2} for any compact Riemannian manifold $N$. By the proof
we know the result still holds for non-compact Riemannian manifold with suitable bounded geometry conditions.
These results will also be applied to construct $\nabla$-martingale with fixed terminal value in
Corollary \ref{c2-2}.

\item [(3)] Theorem \ref{t2-2} provides a systematic way to obtain the existence of a solution
to $N$-valued BSDE, based on which we can apply many results on the $\R^L$-valued BSDE whose generator has quadratic growth
directly. By Theorem \ref{t2-3},
for any compact Riemannian manifold $N$, there exists
a unique global Markovian solution to \eqref{d2-1-1} when the dimension  $m$ of filtering noise is equal to $1$,
which gives us another example about global existence of a solution to
$N$-valued BSDE without any convexity condition. Meanwhile, it also illustrates that for
some BSDE whose generator has quadratic growth, not only the dimension $n$ of solution
(see the difference between scalar valued BSDE and multi-dimensional BSDE), but also the dimension $m$ of filtering noise, will have
crucial effects.

\end{itemize}

The rest of the paper is organized as follows. In Section \ref{section2} we will give a brief introduction
on some preliminary knowledge and notations, including the theory of sub-manifold $N$ in ambient space
$\R^L$. In Section \ref{section2.1}, we are going to summarise our main results and their applications.
 In Section \ref{section3},
the proof of Theorem \ref{t2-2} and \ref{t2-3} will be given. And we will prove Theorem \ref{t2-1} in Section \ref{section4}.

\section{Preliminary knowledge and notations}\label{section2}
\subsection{Sub-manifold of an ambient Euclidean space}\label{section2-1}
Through this paper, suppose that $N$ is an $n$-dimensional compact Riemannian manifold endowed with a Levi-Civita connection
$\nabla$. By the Nash embedding theorem, there exists an isometric embedding $i: N\rightarrow \R^L$
from $N$ to an ambient Euclidean space $\R^L$ with $L>n$. So we could view $N$ as
a compact sub-manifold of $\R^L$. We denote the Levi-Civita connection on $\R^L$ by $\bar \nabla$
(which is the standard differential on $\R^L$).
Let $TN$ be the tangent bundles of $N$ and let $T_p N$ be the tangent space at $p\in N$. For
any $m\in \mathbb{N}_+$, we define
\begin{equation*}
T^{m}N:=\bigcup_{p \in N}(T_p N)^{\otimes m}
\end{equation*}
as the tensor product of $TN$ with order $m$.

For every $p\in N\subset \R^L$, by the Riemannian metric on $N$, we could split $\R^L$ into direct sum as $\R^L=T_p N \oplus T_p^{\bot} N$, where
$T_p^{\bot} N$ denotes orthonormal complement of
$T_p N$. Hence for every $v\in \R^L$ and $p\in N$, we have  a decomposition as follows,
\begin{equation}\label{e2-1}
v=v^{T}+v^{\bot},\ \ v^{T}\in T_p N,\ v^{\bot}\in T_p^{\bot} N,
\end{equation}
we usually call $v^{T}$, $v^{\bot}$ the tangential projection and normal projection of $v\in \R^L$ respectively.

Given smooth vector fields $X,Y$ on $N$, let $\bar X$, $\bar Y$ be the (smooth) extension of $X$, $Y$
on $\R^L$ (which satisfies that $\bar X(p)=X(p)$, $\bar Y(p)=Y(p)$ for any $p\in N$), then we have
$\nabla_X Y(p)=(\bar \nabla_{\bar X}\bar Y)^{T}(p)$, where $(\bar \nabla_{\bar X}\bar Y)^T$ is the tangential projection
defined by \eqref{e2-1}.
Let $A(p):T_pN \times T_p N \to T_p^{\bot} N$ be the second fundamental form at $p\in N$ defined by
\begin{equation}\label{e2-2}
\begin{split}
A(p)(u,v):&=\bar \nabla_{\bar X}\bar Y(p)-\nabla_X Y(p)\\
&=\bar \nabla_{\bar X}\bar Y(p)-(\nabla_{\bar X} {\bar Y})^T(p), \ \ \forall\ 
u,v\in T_p N,
\end{split}
\end{equation}
where $X$, $Y$ are any smooth vector fields on $N$ satisfying $X(p)=u$, $Y(p)=v$,
$\bar X$, $\bar Y$ are any smooth vector fields on $\R^L$ which are extension of $X$ and $Y$ respectively.
The value of
$A(p)(u,v)$ is independent of the choice of extension $X$, $Y$, $\bar X$, $\bar Y$.

We define the distance from $p\in \R^L$ to $N$ as follows
\begin{equation*}
{\rm dist}_N(p):=\inf\{|p-q|;q\in N\subset \R^L\},
\end{equation*}
where $|p-q|$ denotes the Euclidean distance between $p$ and $q$ in $\R^L$. Set
\begin{equation*}
B(N,r):=\{p\in \R^L; {\rm dist}_N(p)<r\},\ \ \forall\ r>0.
\end{equation*}

Since $N$ is compact, it is well known that there exists a $\delta_0(N)>0$ such that
${\rm dist}_N^2(\cdot): B(N,3\delta_0) \to \R_+$ and the nearest projection
map $P_N:B(N, 3\delta_0)\to N$ are smooth, where for every $p\in B(N, 3\delta_0)$, $P_N(p)=q$ with
$q\in N$ being the unique element in $N$ satisfying $|p-q|={\rm dist}_N(p)$.
Moreover, for every $p\in B(N, 3\delta_0)$, suppose $\gamma:[0,{\rm dist}_N(p)]\to \R^L$ is the
unique unit speed geodesic in $\R^L$ (which is in fact a straight line) such that $\gamma(0)=P_N(p)$, $\gamma({\rm dist}_N(p))=p$, then
for every $p\in B(N, 3\delta_0)$ it holds
\begin{equation}\label{e2-2a}
\begin{split}
&\bar \nabla {\rm dist}_N (P_N(p))=\gamma'(0)\in T_{P_N(p)}^{\bot}N,\\
&\bar \nabla {\rm dist}_N(p)=\gamma'({\rm dist}_N(p))=\gamma'(0),\\
&|\bar \nabla {\rm dist}_N(p)|=1. 
\end{split}
\end{equation}
Here we have used property $\gamma'(0)=\gamma'({\rm dist}_N(p))$ since $\gamma(\cdot)$ is a straight line in $\R^L$.
Moreover, we still have the following characterization for second fundamental form $A$,
\begin{equation}\label{e2-3a}
A(p)(u,u)=\sum_{i,j=1}^L \frac{\partial^2 P_N}{\partial p_i \partial p_j}(p)u_iu_j,\
 p\in N,\ u=(u_1,\cdots,u_L)\in T_p N.
\end{equation}

We choose a cut-off function $\phi\in C^{\infty}(\R,\R)$ such that
\begin{equation*}
\phi(s)=
\begin{cases}
& 1,\ \ s<\delta_0,\\
&\in (0,1),\ \ s\in [\delta_0,2\delta_0],\\
&0,\ \ \ s>2\delta_0.
\end{cases}
\end{equation*}
It is easy to verify that $p\mapsto \phi({\rm dist}_N(p))$ is a smooth function on $\R^L$.
Then we could extend the second fundamental form $A$ defined by \eqref{e2-2} to
$\bar A:\R^L\to L(\R^L \times \R^L ; \R^L) $ (here $L(\R^L \times \R^L ; \R^L)$ denotes the collection of all
linear maps from $\R^L\times \R^L$ to $\R^L$) as follows
\begin{equation}\label{e2-3}
\begin{split}
\bar A(p)(u,u):=
\begin{cases}
&\phi\big({\rm dist}_N(p)\big)\sum_{i,j=1}^L \frac{\partial^2 P_N}{\partial p_i \partial p_j}(P_N(p))u_iu_j,\ \ \  p\in B(N,2\delta_0),\\
&0,\ \ \ p\in \R^L/B(N,2\delta_0)
\end{cases}
\end{split}
\end{equation}
for all $u\in \R^L$. According to
\eqref{e2-3a}, \eqref{e2-3} and the definition of $\phi$, we know immediately that $\bar A$ is a smooth map and
\begin{equation*}
\begin{split}
&\bar A(p)(u,v)=A(p)(u,v),\ \forall\ p\in N,\ u,v\in T_p N,\\
&\bar A(p)=0,\ \forall\ p\in \R^L/B(N,2\delta_0).
\end{split}
\end{equation*}
We refer readers to \cite[Section III.6]{C}, \cite[Chapetr 6]{D} or \cite[Section 1.3]{LW} for detailed introduction
 concerning about various properties for sub-manifold $N$ of $\R^L$.

\subsection{Non-linear generator $f$}
\emph{In this paper, we always make the following assumption for $f$.}
\begin{assumption}\label{a2-1}
Suppose that $f: N\times T^m N \to TN$ is a $C^1$ map such that $f(p,u)\in T_p N$ for
every $p\in N$, $u=(u_1,\cdots,u_m)\in T_p^m N$.
And there exists a $C_0>0$ such that for every
$p\in N$, $u\in T^m_p N$,
\begin{equation}\label{e3-0a}
|f(p,u)|_{T_p N}\le C_0(1+|u|_{T_p N}),\ |\nabla_p f(p,u)|_{T_p N}+|\nabla_u f(p,u)|_{T_p N}\le C_0,
\end{equation}
where $\nabla_p$ and $\nabla_u$ denote the covariant derivative with respect to the variables $p$ in $N$ and $u$ in $T^m N$ respectively.
\end{assumption}

Now we define a $C^1$ extension $\bar f:\R^L\times \R^{m L}\to \R^L$ of $f$ as follows
\begin{equation}\label{e2-5}
\bar f(p,u):=
\begin{cases}
& \phi\big({\rm dist}_N(p)\big)f\big(P_N(p),\Pi_N(P_N(p))u\big),\ p\in B(N,2\delta_0),\\
& 0,\ \ p\in \R^L/B(N,2\delta_0).
\end{cases}
\end{equation}
Here $\phi:\R\to \R$, $P_N:B(N,2\delta_0)\to N$ are the same as those in \eqref{e2-3} and
$\Pi_N(p):\R^L \to T_p N$ denotes the projection map to $T_p N$ defined by \eqref{e2-1} for every
$p \in N$.  

Note that $N$ is compact, combing \eqref{e2-5} with \eqref{e3-0a}
we obtain immediately following estimates for the extension $\bar f:\R^L\times \R^{m L}\to \R^L$ of $f$.
\begin{equation}\label{e3-1a}
|\bar f(p,u)|+|\bar \nabla_p \bar f(p,u)|\le C_1(1+|u|), |\bar \nabla_u \bar f(p,u)|\le C_1,\ \ \forall\ p\in \R^L, u\in \R^{m L},
\end{equation}
Here $\bar \nabla_p$ and $\bar \nabla_u$ denote the gradient in $\R^L$ with respect to variables
$p$ and $u$ respectively.

\subsection{Space of Malliavin differentiable random variables}

Through this paper, we will fix a probability space $(\Omega,\scr{F},\P)$ and an $\R^m$-valued standard Brownnian motion
$\{B_t=(B_t^1,\cdots,B_t^m)\}_{t\ge 0}$ on  $(\Omega,\scr{F},\P)$ with some $m\in \Z_+$. Let $\{\scr{F}_t\}_{t\ge 0}$ denote
the natural filtration associated with $\{B_t\}_{t\ge 0}$. For simplicity we call a process adapted (or predictable)
when it is adapted (or predictable) with respect to the filtration $\{\scr{F}_t\}_{t\ge 0}$.

Set
\begin{equation}\label{e2-4}
\begin{split}
\F C_b^\infty(\R^L):=&\Big\{\xi(\omega)=\left(\xi^1(\omega),\cdots,\xi^L(\omega)\right)\Big|
\xi^i(\omega)=g^i\left(B_{t_{i1}},\cdots, B_{t_{ik_i}}\right),\ \forall\ 1\le i \le L\\
& {\rm for\ some}\ g^i\in C_b^\infty(\R^{mk_i};\R),\ k_i\in \mathbb{N}_+,\ 0<t_{i1}<\cdots<t_{ik_i}\Big\}.
\end{split}
\end{equation}
Let $\mathbb{D}: \F C_b^\infty(\R^L)\to L^2(\Omega; L^2([0,T];\R^m\times \R^L);\P)$ be the gradient operator such that
for every $\xi\in \F C_b^\infty(\R^L)$ with expression \eqref{e2-4} and non-random  $\eta\in L^2([0,T];\R^m)$,
\begin{align*}
&\D \xi(\omega)(t)=\left(\D \xi^1(\omega)(t),\cdots,\D \xi^L(\omega)(t)\right),\\
& \quad \quad \int_0^T \D\xi^i(\omega)(t)\cdot \eta(t)dt\\
&=\lim_{\e \to 0}
\frac{g^i\left(B_{t_{i1}}+\e\int_0^{t_{i1}}\eta(s)\,ds,\cdots, B_{t_{ik_i}}+\e\int_0^{t_{ik_i}}\eta(s)\,ds\right)-g^i\left(B_{t_{i1}},\cdots, B_{t_{ik_i}}\right)}{\e},
\ 1\le i\le L,
\end{align*}
where $\cdot$ denotes the inner product in $\R^m$.

For every $\xi\in \F C_b^\infty(\R^L)$, we define
\begin{align*}
\|\xi\|_{1,2}^2:=\E[|\xi|^2]+\E\left[\int_0^T |\D \xi(t)|^2 dt\right].
\end{align*}
Let $\mathscr{D}^{1,2}(\R^L):=\overline{\F C_b^\infty(\R^L)}^{\|\cdot\|_{1,2}}$ be the completion of
$\F C_b^\infty(\R^L)$ with respect to norm $\|\cdot\|_{1,2}$. It is well known that
$(\D, \F C_b^\infty(\R^L))$ could be extended to a closed operator
$(\D,\mathscr{D}^{1,2}(\R^L))$.

We define the space of $N$-valued Malliavin differentiable random variables as follows
\begin{align*}
\mathscr{D}^{1,2}(N):=\{\xi\in \mathscr{D}^{1,2}(\R^L); \xi(\omega)\in N\ {\rm for\ a.s.}\ \omega\in \Omega \}.
\end{align*}

We refer readers to the monograph \cite{N} for detailed introduction on the theory of Malliavin calculus.

\subsection{Other notations} We use := as a way of definition. Let $\mathbb{T}^m=\R^m/\mathbb{Z}^m$ be the $m$-dimensional torus.
For every $x\in \T^m$, $p\in \R^L$ and $r>0$, set $B_{\T^m}(x,r):=\{y\in \T^m; |y-x|<r\}$
and $B(p,r):=\{q\in \R^L;
|q-p|<r\}$. Let $dt$ and $dx$ be the Lebesgue measure on $[0,T]$ and $\T^m$ respectively.
We denote the derivative, gradient and Laplacian with respect to the variable $x\in \T^m$
by $\partial_{x_i}$, $\nabla_x$ and $\Delta_x$ respectively. The covariant derivative for the variable in
$N$ is denoted by $\nabla$, while we use $\bar \nabla$ and $\bar \nabla^2$ to represent the first and second order
gradient operator in $\R^L$ respectively.
\emph{We use $\langle, \rangle$ to denote
both the Riemannian metric on $TN$ and the Euclidean inner product on $\R^L$ (note that for every $p\in N$ and $u,v\in T_p N$,
we have $\langle u,v\rangle_{T_p N}=\langle u,v\rangle_{\R^L}$). Meanwhile let $\cdot$ denote the inner product in
$\R^m$ (in the tangent space of $\T^m$).} Without extra emphasis, we use a.s. and a.e. to mean almost sure with respect to
$\P$ and almost every where with respect to Lebesgue measure on $\T^m$ respectively. \emph{Throughout the paper, the constant
$c_i$ will be independent of $\e$.} For any $q\ge 1$ and $k\in \mathbb{N}_+$, set
\begin{equation*}
\begin{split}
&C^k(\T^m;N):=\{u\in C^k(\T^m;\R^L);u(x)\in N\ {\rm for\ every}\ x\in \T^m\},\\
&L^q(\T^m;\R^L):=\left\{u:\T^m\to \R^L; \|u\|_{L^q(\T^m;\R^L)}^q:=\int_{\T^m} |u(x)|^q dx<\infty\right\},\\
&L^q(\T^m;N):=\{u\in L^q(\T^m;\R^L); u(x)\in N \ {\rm for}\ {\rm a.e.}\ x\in \T^m\}.\\
\end{split}
\end{equation*}

\section{Main theorems and their applications}\label{section2.1}

\subsection{$N$-valued BSDE}

In this subsection we are going to give the definition of $N$-valued BSDE through the BSDE on ambient space $\R^L$.
Fixing a time horizon $T\in (0,\infty)$, $m\in \mathbb{N}_+$ and $q\in (1,\infty)$, we define
\begin{equation*}
\begin{split}
\mathscr{S}^q(\R^L):=& \Big\{Y:[0,T]\times \Omega \to \R^L; Y\ \text{is}\ \text{predictable},\ \E\Big[\sup_{t\in [0,T]}|Y_t|^q\Big]<\infty,\\
&\ t\mapsto Y_t(\omega)\ \text{is\ continuous \ on }[0,T]\ \text{for}\ a.s.\ \omega\in \Omega\Big\},\\
\mathscr{S}^q(N):=& \Big\{Y\in \mathscr{S}^q(\R^L); \text{for\ any}\ t\in [0,T],\ Y_t\in N\ a.s. \Big\}.
\end{split}
\end{equation*}
\begin{equation*}
\begin{split}
\M^q_m(\R^L):=& \Big\{Z:[0,T]\times \Omega \to \R^{m L}; Z\ \text{is}\ \text{predictable},\
\E\Big[\Big(\int_0^T |Z_t|^2 dt\Big)^{q/2}\Big]<\infty \Big\}.
\end{split}
\end{equation*}

We usually write the components of a $Z\in \M^q_m(\R^L)$ by
$Z_t(\omega)=\big\{Z_t^{i,j}(\omega);1\le i \le m; 1\le j \le L\big\}$
and set
$$Z_t^i(\omega)=(Z_t^{i,1}(\omega),\cdots, Z_t^{i,L}(\omega))\in \R^L,\ \forall\ t\in [0,T],\ 1\le i \le m, \omega\in \Omega.$$
Let  
\begin{equation*}
\begin{split}
\S^q\oplus\M^q_m(N):=& \Big\{(Y,Z); Y\in \S^q(N),\ Z\in \M^q_m(\R^L),\
\\
& \text{and}\ Z_t^i\in T_{Y_t} N\ \text{for}\ dt\times \P\ a.s.\ (t,\omega)\in
[0,T]\times \Omega,\ 1\le i \le m\Big\}.
\end{split}
\end{equation*}

\begin{definition}\label{d2-1}
We call a pair of process $(Y,Z)$ is a solution of $N$-valued
BSDE \eqref{d2-1-1} if $(Y,Z)\in \S^q\oplus\M^q_m(N)$ for some $q\ge 2$ and satisfies
the following equation in $\R^L$ (where $(Y,Z)$ is viewed as an $\R^L\times \R^{mL}$-valued process)
\begin{equation}\label{d2-1-1}
Y_t=\xi-\sum_{i=1}^m \int_t^T Z_s^i dB_s^i-\sum_{i=1}^m\frac{1}{2}\int_t^T \bar A(Y_s)(Z_s^i,Z_s^i)ds
+\int_t^T  \bar f (Y_s,Z_s)ds.
\end{equation}
Here $\xi:\Omega \to N\subset \R^L$ is an $N$-valued $\scr{F}_T$ measurable random variable,
$\bar A:\R^L \to L(\R^L\times \R^L;\R^L)$ and $\bar f:\R^L\times \R^{mL}\to \R^L$ are defined by \eqref{e2-3} and
\eqref{e2-5} respectively.
\end{definition}

\begin{rem}\label{r2-2}
Let $(U,\varphi)$ be a local coordinate on $N$ such that $U\subset N$ and $\varphi:U \rightarrow \varphi(U)\subset \R^n$ is
a smooth diffeomorphism. Suppose that $(Y,Z)$ is a solution of $N$-valued BSDE \eqref{d2-1-1} with
$Y$ always staying in $U$. Then by applying It\^o formula to
$\varphi(Y_t)$ (by the same computation in the proof of Proposition \ref{p2-1} below) it is
not difficult to verify that $\left(\varphi(Y),d\varphi(Y)(Z)\right)$ is a solution of \eqref{e1-3}
defined by \cite{B1,B2} with $\Gamma_{ij}^k$ being the Christoffel symbols associated with
Levi-Civita connection $\nabla$, where $d\varphi: TN \rightarrow \R^n$ denotes the tangential map of $\varphi:U\subset N\to \R^n$.
\end{rem}

\begin{rem}
Note that the fundamental form $A$ in \eqref{d2-1-1} will depend on
the Riemmannian metric (due to the decomposition of tangential direction and
normal direction) and associated Levi-Civita connection $\nabla$ on $N$. But we are not sure
whether Definition \ref{d2-1} could be extended to the case that $N$ is only a smooth
manifold endowed with an affine connection.
\end{rem}

Now we will give the following result about the relation between a solution of $N$-valued BSDE and
a general $\R^L$-valued solution of BSDE \eqref{d2-1-1}.

\begin{theorem}\label{t2-2}
Suppose $Y\in \S^q(\R^L)$, $Z\in \M_m^q(\R^L)$ with some $q\ge 2$ and $m\ge 1$ is an
$\R^L$-valued solution of
 BSDE \eqref{d2-1-1}
which satisfies that
\begin{equation}\label{t2-2-2}
|Z_t(\omega)|\le C_2,\ dt\times\P-{\rm a.e.}\ (t,\omega)\in [0,T]\times \Omega,
\end{equation}
for some $C_2>0$. If we also assume that the terminal value $\xi\in N\subset \R^L$ a.s. in \eqref{d2-1-1},
then $(Y,Z)$ is a  solution of $N$-valued BSDE \eqref{d2-1-1}.
\end{theorem}

With Theorem \ref{t2-2}, we can obtain the existence of a unique $N$-valued solution of \eqref{d2-1-1}
by several known results on the $\R^L$-valued solution of general BSDE whose generator has quadratic growth.

\begin{corollary}\label{c2-1}
Suppose $\xi\in \mathscr{D}^{1,2}(N)$ and
\begin{equation}\label{t2-2-1}
|\D \xi(\omega)(t)|\le C_3, dt\times\P-{\rm a.e.}\ (t,\omega)\in [0,T]\times \Omega.
\end{equation}
Then we can find a positive constant $T_0=T_0(C_3)$
such that there exists a unique  solution $(Y,Z)$ to  $N$-valued BSDE \eqref{d2-1-1}
in
time interval $[0,T_0]$ (with terminal value $\xi$) which satisfies
\eqref{t2-2-2}
for some $C_2>0$.
\end{corollary}
\begin{proof}
According to \eqref{e2-3} and \eqref{e3-1a} we have
for every $y_1,y_2\in \R^L$ and $z_1,z_2\in \R^{m L}$,
\begin{equation}\label{t2-2-3}
\begin{split}
&\left|\bar A(y_1)(z_1,z_1)-\bar A(y_2)(z_2,z_2)\right|\le c_1(1+|z_1|^2+|z_2|^2)\left(|y_1-y_2|+|z_1-z_2|\right),\\
&\left|\bar f(y_1,z_1)-\bar f(y_2,z_2)\right|\le c_1(1+|z_1|+|z_2|)\left(|y_1-y_2|+|z_1-z_2|\right).
\end{split}
\end{equation}

 Based on \eqref{t2-2-1} and \eqref{t2-2-3}, if we view \eqref{d2-1-1} as an $\R^L$-valued BSDE,
 by \cite[Theorem 3.1]{KLT} or \cite[Theorem 2.1]{HR} (although \eqref{t2-2-3} is slightly different from
 those in \cite{KLT} where associated coefficients are required to be globally Lipschitz continuous with respect to variable $y$, following the same
 procedure in the proof of \cite[Theorem 3.1]{KLT} we can still obtain the desired conclusion here, see also
 the arguments in \cite[Example 2.2]{KLT})
we can find a $T_0>0$ such that there exists a unique solution $(Y,Z)$ with $Y\in \scr{S}^4(\R^L)$, $Z\in \M^4_m(\R^L)$ to \eqref{d2-1-1}
in time interval $t\in [0,T_0]$ which satisfies \eqref{t2-2-2} for some $C_2>0$.

Then applying Theorem \ref{t2-2} we obtain the desired conclusion immediately.
\end{proof}

Similarly, according to \cite{XZ}, under some condition on the existence of a Lyapunov function, we can
also obtain the unique existence of a global  Markovian solution of $N$-valued BSDE \eqref{d2-1-1} by applying Theorem \ref{t2-2}, and
we omit the details here.

Moreover, without any convexity condition (such as the existence of Lyapunov function or doubly convex function), we also have
the unique existence of global Markovian solution of $N$-valued BSDE \eqref{d2-1-1} when $m=1$.

\begin{theorem}\label{t2-3}
Assume $m=1$. Given an arbitrary $T>0$, suppose $\xi=h(B_T)$ for some $h\in C^1(\T^m;N)$ in \eqref{d2-1-1}
(since we could also view $h\in C^1(\T^m;N)$ as a function $h\in C^1(\R^m;N)$, $h(B_T)$ is well defined here).
Then there exists a unique solution $(Y,Z)$ of
$N$-valued BSDE \eqref{d2-1-1}
in time interval $[0,T]$ which satisfies
\eqref{t2-2-2} for some $C_2>0$.
\end{theorem}
\subsection{$L^2(\T^m;N)$-valued BSDE}

Still for a given time horizon $T\in (0,\infty)$, we define
\begin{equation*}
\begin{split}
\scr{S}^2(\T^m;N):=\Big\{& Y:[0,T]\times \Omega \to L^2(\T^m;N); Y\ {\rm is \ predictable},\\
 &t\mapsto Y_t(\omega)\ {\rm is\ continuous\ in}\  L^2(\T^m;\R^L)\ {\rm for\ a.s.}\ \omega\in \Omega,\\
& \|Y\|_{\scr{S}^2(\T^m;\R^L)}^2:=\E\big[\sup_{t\in [0,T]}\|Y_t\|_{L^2(\T^m;\R^L)}^2\big]<\infty\Big\},
\end{split}
\end{equation*}
\begin{equation*}
\begin{split}
\M^2(\T^m;\R^L):=\Big\{ & Z:[0,T]\times \Omega \to  L^2(\T^m; \R^{mL});\ Z\ {\rm is \ predictable},\\
&\|Z\|_{\M^2(\T^m;\R^L)}^2:=\E\Big[\int_0^T \|Z_s\|_{L^2(\T^m;\R^{mL})}^2ds\Big]<\infty\Big\}.
\end{split}
\end{equation*}
Indeed, if
$\sup_{t\in [0,T]}\|Y_t-\tilde Y_t\|_{L^2(\T^m;\R^L)}=0$  a.s. 
for some $Y,\tilde Y\in \scr{S}^2(\T^m;N)$, then we view $Y$ and $\tilde Y$ as the same element in
$\scr{S}^2(\T^m;N)$. Similar equivalent relations also hold for the elements $Z$ in $\M^2(\T^m;\R^L)$.

We usually write the components of a $Z\in \M^2(\T^m;\R^L)$ by
$Z_t^x(\omega)=\big\{Z_t^{x,i,j}(\omega); 1\le i \le m, 1\le j \le L\big\}$ for any
$t\in [0,T]$, $x\in \R^d$ and $\omega\in \Omega$. We also set
$$Z_t^{x,i}(\omega)=(Z_t^{x,i,1}(\omega),\cdots, Z_t^{x,i,L}(\omega))\in \R^L,\ \forall\ t\in [0,T],\ x\in \R^m,\ 1\le i \le m,\ \omega\in \Omega.$$
Let
\begin{equation*}
\begin{split}
& \scr{S}\otimes \M^2(\T^m;N):=\Big\{(Y,Z); Y\in \scr{S}^2(\T^m;N), Z\in \M^2(\T^m;\R^L),\\
& \ \ \ \text{and}\ Z_t^{x,i}\in T_{Y_t^x} N\ \text{for}\ dt\times dx\times \P-{\rm a.e.}\ (t,x,\omega)\in
[0,T]\times \T^m\times \Omega,\ 1\le i \le m\Big\}.
\end{split}
\end{equation*}
Now we can give the definition of $L^2(\T^m;N)$-valued (weak) solution of a BSDE,
\begin{definition}\label{d2-2}
We call a pair of process $(Y,Z)$ is a solution of
$L^2(\T^m;N)$-valued BSDE \eqref{d2-2-2} if we can find an equivalent version of $(Y,Z)\in \scr{S}\otimes \M^2(\T^m;N)$
(still denoted by $(Y,Z)$ for simplicity of notation) such that
for a.e. $x\in \T^m$ the following
equation holds for every $t\in [0,T]$,
\begin{equation}\label{d2-2-2}
Y_t^x=h(B_T+x)-\sum_{i=1}^m \int_t^T Z_s^{x,i} dB_s^i-\sum_{i=1}^m\frac{1}{2}\int_t^T \bar A(Y_s^x)(Z_s^{x,i},Z_s^{x,i})ds
+\int_t^T \bar f (Y_s^x,Z_s^x)ds.
\end{equation}
Here $h:\T^m \to N$ is an $N$-valued non-random function, 
$\bar A:\R^L \to L(\R^L\times \R^L;\R^L)$ and $\bar f:\R^L\times \R^{m L}\to \R^L$ are defined by \eqref{e2-3} and
\eqref{e2-5} respectively.
\end{definition}

Now we will give the following results concerning about the global existence of a solution of
$L^2(\T^m;N)$ valued BSDE \eqref{d2-2-2} for an arbitrarily fixed compact Riemannian manifold $N$.

\begin{theorem}\label{t2-1}
Suppose $h\in C^1(\T^m;N)$, then for any $T>0$,
there exists a solution $(Y,Z)$ of $L^2(\T^m;N)$-valued BSDE \eqref{d2-2-2} in time interval $t\in [0,T]$.
\end{theorem}

\begin{rem}
Intuitively, the global solution of $L^2(\T^m;N)$-valued BSDE \eqref{d2-2-2}  always exists
for any compact Riemannian manifold $N$ (without any other convexity condition) since the
collection $\Xi_0:=\{x\in \T^m; |Z_t^x|=+\infty\ \text{for\ some\ }\ t\in [0,T]\}$ is a Lebesgue-null set
in $\T^m$ (which could be seen in the proof of Theorem \ref{t2-1}).

Meanwhile, due to the lack of monotone condition on the generator
(see the corresponding monotone conditions in \cite{BM,MX,ZZ2,ZZ3}), it seems difficult to prove
the uniqueness of the solution of $L^2(\T^m;N)$-valued BSDE \eqref{d2-2-2}.
\end{rem}

\begin{rem}
The exceptional Lebesgue null set for $x\in \T^m$ in \eqref{d2-2-2} may depend on the choice of $h$.
We do not know whether we can find a common null set $\Xi$ which ensures \eqref{d2-2-2} valid for
every $h\in C^1(\T^m;N)$ and $x\notin \Xi$.
\end{rem}

We also have the following characterization for solution of $L^2(\T^m;N)$-valued BSDE.
\begin{proposition}\label{r2-1}
$(Y,Z)$ is an 
solution of
$L^2(\T^m;N)$-valued BSDE \eqref{d2-2-2} if and only if 
$(Y,Z)\in \scr{S}\otimes \M^2(\T^m;N)$ and
for every $\psi\in C^2(\T^m;\R^L)$ and $t\in [0,T]$ there exists a $\P$-null set $\Pi_0$ such that
for all $\omega\notin \Pi_0$, it holds that
\begin{equation}\label{r2-1-1}
\begin{split}
&\int_{\T^m}\langle Y_t^x, \psi(x)\rangle dx=\int_{\T^m}\langle h(B_T+x), \psi(x)\rangle dx-
\sum_{i=1}^m  \int_t^T\Big(\int_{\T^m}\langle Z_s^{x,i}, \psi(x)\rangle dx \Big)dB_s^i\\
&  -\sum_{i=1}^m \frac{1}{2} \int_t^T \int_{\T^m}\langle\bar A(Y_s^x)(Z_s^{x,i},Z_s^{x,i}),\psi(x)\rangle dx ds
+\int_t^T \int_{\T^m}\langle \bar f (Y_s^x,Z_s^x), \psi(x)\rangle dx ds.
\end{split}
\end{equation}
\end{proposition}
\begin{proof}
If \eqref{d2-2-2} holds for a.e. $x\in \T^m$, obviously we can verify \eqref{r2-1-1}.

Now we assume that there exists a $(Y,Z)\in \scr{S}\otimes \M^2(\T^m;N)$ such that
\eqref{r2-1-1} holds a.s. for every $\psi\in C^2(\T^m;\R^L)$ and $t\in [0,T]$. Since there exists a
countable dense subset $\Theta \subset C^2(\T^m;\R^L)$ of $L^2(\T^m;\R^L)$ under $L^2$ norm, we can find
a Lebesgue null set $\Xi_1\subset \T^m$ and a $\P$-null set $\Pi\subset \Omega$ such that
\eqref{d2-2-2} holds for every $\omega\notin \Pi$, $x\notin \Xi_1$ and $t\in [0,T]\cap \mathbb{Q}$,
where
$\mathbb{Q}$ denotes the collection of all rational numbers.

Note that we have $\E\left[\int_0^T \int_{\T^m}|Z_t^x|^2 dx dt\right]<\infty$ by definition of
$\M^2(\T^m;\R^L)$. Hence there exists a Lebesgue null set $\Xi_2\subset \T^m$, such that
\begin{equation*}
\E\left[\int_{0}^T |Z_t^x|^2dt\right]<\infty,\ \ \forall\ x\notin \Xi_2.
\end{equation*}
This, along with \eqref{e3-1a} implies immediately that for every  $\omega\notin \Pi$ and
$x\notin \Xi_1\cup \Xi_2$, the function $t\mapsto \sum_{i=1}^m \int_t^T Z_s^{x,i}dB_s+\sum_{i=1}^m
\frac{1}{2}\int_t^T\bar A(Y_s^x)\left(Z_s^{x,i},Z_s^{x,i}\right)ds-\int_t^T \bar f\left(Y_s^x,Z_s^x\right)ds$
is continuous in interval $[0,T]$.
So for every $\omega\notin \Pi$, $x\notin \Xi_1\cup \Xi_2$ and
$t\in [0,T]$ we can define
\begin{align*}
\quad\quad \hat Y_s^x(\omega):=\lim_{s\to t;s\in \mathbb{Q}}&\Bigg(h(B_T+x)-\sum_{i=1}^m \Big(\int_s^T Z_r^{x,i}dB_r+
\frac{1}{2}\int_s^T\bar A(Y_r^x)\left(Z_r^{x,i},Z_r^{x,i}\right)dr\Big)\\
&+\int_s^T \bar f\left(Y_r^x,Z_r^x\right)dr\Bigg).
\end{align*}

Set
\begin{equation*}
\tilde Y_t^x(\omega):=
\begin{cases}
& Y_t^x(\omega),\ \ \ \ \ \ \ \ \ \ \ \ \ \ \ {\rm if}\ t\in [0,T]\cap \mathbb{Q}, x\notin \Xi_1\cup\Xi_2, \omega\notin \Pi,\\
& \hat Y_t^x(\omega),
\ \ \ \ \ \ \ \ \ \ \ \ {\rm if}\ t\in [0,T]\cap \mathbb{Q}^c, x\notin \Xi_1\cup\Xi_2, \omega\notin \Pi,\\
&0,\ \ \ {\rm otherwise}.
\end{cases}
\end{equation*}
Then by definition it is easy to verify that $(\tilde Y^x,Z^x)$ satisfies \eqref{d2-2-2}
for every $x\notin \Xi_1\cup \Xi_2$, $\omega\notin \Pi$ and $t\in [0,T]$.

Still by definition of $\tilde Y$, we have
$Y_t^x(\omega)=\tilde Y_t^x(\omega)$ for every $\omega\notin \Pi$, $t\in [0,T]\cap \mathbb{Q}$ and $x\notin \Xi_1\cup \Xi_2$.
Meanwhile due to $Y\in \mathscr{S}^2(\T^m;N)$ there exists a $\P$-null set $\Pi_0$ such that $t\mapsto Y_t^{\cdot}(\omega)$ is
continuous in $L^2(\T^m;\R^L)$ for every $\omega\notin \Pi_0$. This along with the definition of
$\tilde Y$ implies immediately that given any $\omega\notin \Pi\cup\Pi_0$ and $t\in [0,T]\cap \mathbb{Q}^c$,
$Y_t^{x}(\omega)=\tilde Y_t^x(\omega)$ ($=L^2$-$\lim_{s\to t;s\in \mathbb{Q}}Y_s^{\cdot}(\omega)$) for a.e. $x\in \T^m$
(the exceptional set for $x\in \T^m$ may depend on $t$). Combing all the properties above we arrive at
\begin{align*}
\sup_{t\in [0,T]}\|Y_t(\omega)-\tilde Y_t(\omega)\|_{L^2(\T^m;\R^L)}=0,\ \forall\ \omega\notin \Pi\cup\Pi_0.
\end{align*}
Hence $Y$ and $\tilde Y$ is the same element in $\mathscr{S}^2(\T^m;N)$, so we can find an equivalent version
$(\tilde Y,Z)$ of $(Y,Z)$ which satisfies \eqref{d2-2-2} a.s. for
each $x\notin \Xi_1\cup \Xi_2$.
\end{proof}

\subsection{Existence of $\nabla$-martingale with fixed terminal value}

In this subsection we will give an application of Theorem \ref{t2-3} and \ref{t2-1}
on the construction of $\nabla$-martingales, which also
illustrates that Definition \ref{d2-1} and \ref{d2-2} are natural for an $N$-valued BSDE.

\begin{proposition}\label{p2-1}
\begin{itemize}
\item [(1)] Suppose 
$(Y,Z)$ is a solution of $N$-valued BSDE \eqref{d2-1-1}.
For every
$g\in C^2(N;\R)$ and $t\in [0,T]$, let
\begin{equation}\label{p2-1-1}
M_t^g:=g(Y_t)-g(Y_0)-\sum_{i=1}^m\frac{1}{2}\int_0^t {\rm Hess}g(Y_s)(Z_s^i,Z_s^i)ds+\int_0^t \langle \nabla g(Y_s), f(Y_s,Z_s)\rangle ds,
\end{equation}
where ${\rm Hess}$ denotes the Hessian operator on $N$ associated with Levi-Civita connection $\nabla$.
Then $\{M_t^g\}_{t\in [0,T]}$ is a local martingale.

\item [(2)] Suppose $(Y,Z)$ is a solution of $L^2(\T^m;N)$-valued BSDE \eqref{d2-2-2}. Given some $g\in C^2(N;\R)$ and $x\in \T^m$ we define $\{M_t^{g,x}\}_{t\in [0,T]}$ by the same way of \eqref{p2-1-1} with $(Y_t,Z_t)$ replaced by
$(Y_t^x,Z_t^x)$. Then there exists a Lebesgue-null set $\Xi \subset \T^m$ such that
$\{M_t^{g,x}\}_{t\in [0,T]}$ is a local martingale for every $g\in C^2(N;\R)$ and $x\notin \Xi$.
\end{itemize}
\end{proposition}
\begin{proof}
We only prove part (1) of desired conclusion. The part (2) could be proved by
applying \eqref{d2-2-2} and the same procedures for the proof of (1).

By the same way of \eqref{e2-5}, we could extend $g$ to a $C^2$ function
$\bar g:\R^L \to \R$ with compact support. Since we could still view $(Y,Z)$ as an $\R^L$-valued solution to
\eqref{d2-1-1}, applying It\^o formula to $\bar g$ we obtain that the process $\{\bar M_t^{\bar g}\}_{t\in [0,T]}$ defined by
\begin{equation}\label{p2-1-2}
\begin{split}
\bar M_t^{\bar g}:=&\bar g(Y_t)-\bar g(Y_0)-\sum_{i=1}^m
\frac{1}{2}\int_0^t \Big(\bar \nabla^2 \bar g(Y_s)\big(Z_s^i,Z_s^i\big)+
\left\langle \bar \nabla \bar g(Y_s), \bar A(Y_s)(Z_s,Z_s)\right\rangle\Big)ds\\
&+
\int_0^t \langle \bar \nabla \bar g(Y_s), \bar f(Y_s,Z_s)\rangle ds
\end{split}
\end{equation}
is a local martingale. 

For every $p\in N$, $u\in T_p N$, let $X$, $\bar X$ be arbitrarily fixed smooth vector fields on $N$ and $\R^L$ satisfying
$X(p)=\bar X(p)=u$, so by \eqref{e2-2} 
we have
\begin{align*}
&\quad \bar \nabla^2 \bar g(p)(u,u)+\langle \bar \nabla \bar g(p), \bar A(p)(u,u)\rangle \\
&=\bar X\left(\langle \bar \nabla \bar g, \bar X\rangle\right)(p)-
\langle \bar \nabla \bar g(p), \bar \nabla_{\bar X}\bar X(p)\rangle+
\langle \bar \nabla \bar g(p), \bar \nabla_{\bar X}\bar X(p)-\nabla_{X} X(p) \rangle\\
&=\bar X\left(\langle \bar \nabla \bar g, \bar X\rangle\right)(p)-\langle \bar \nabla \bar g(p), \nabla_{X} X(p) \rangle
\\
&=X\left(\langle \nabla g, X\rangle\right)(p)-\langle \nabla  g(p), \nabla_{X} X(p) \rangle\\
&={\rm Hess}g(p)(X(p),X(p))= {\rm Hess}g(p)(u,u).
\end{align*}
Here in the third step above we have applied the property that $\langle \bar \nabla \bar g(p),\bar X(p)\rangle=
\langle \nabla g(p), X(p)\rangle$ for every $p\in N$ due to $(\bar \nabla \bar g(p))^{T}=\nabla g(p)$. Similarly for
every $p\in N$ and $u\in T_p^m N$ (note that $f(p,u)\in T_p N$) we obtain
\begin{align*}
\langle \bar \nabla \bar g(p), \bar f(p,u)\rangle=\langle \nabla g(p), f(p,u)\rangle.
\end{align*}
Combing all above properties with the fact that $Y_t\in N$ a.s. for every $t\in [0,T]$,
$Z_t\in T_{Y_t} N$ for $dt\times \P$-a.s. $(t,\omega)\in [0,T]\times \P$ into \eqref{p2-1-2} yields
that $\bar M_t^{\bar g}$=$M_t^g$ a.s. for every $t\in [0,T]$. Therefore we know immediately that $M_t^g$ is a local martingale.
\end{proof}

 Recall that we call the adapted process
 $\{X_t\}_{t\in [0,T]}$ a $\nabla$-martingale if it is an $N$-valued semi-martingale and for
 every $g\in C^2(N;\R)$,
 \begin{equation*}
 M_t^g:=g(X_t)-g(X_0)-\frac{1}{2}\int_0^t {\rm Hess}g(X_s)\left(dX_s,dX_s\right)
 \end{equation*}
 is a local martingale. Here $(dX_t,dX_t)$ denotes the quadratic variation for $X_t$.

Then taking $f\equiv 0$, combing Theorem \ref{t2-3}, Theorem \ref{t2-1} and Proposition \ref{p2-1} together we could
obtain the following results concerning about the existence of $\nabla$-martingale on $N$ with fixed terminal value
in arbitrary time interval immediately.
\begin{cor}\label{c2-2}
Suppose $h\in C^1(\T^m;N)$ and $T>0$, then the following statements hold.
\begin{itemize}
\item [(1)] For a.e. $x\in \T^m$, there exists a $\nabla$-martingale $\{Y_t\}_{t\in [0,T]}$
with terminal value $Y_T=h(B_T+x)$.

\item [(2)] If $m=1$, then there exists a $\nabla$-martingale $\{Y_t\}_{t\in [0,T]}$
with terminal value $Y_T=h(B_T)$.
\end{itemize}
\end{cor}

\section{The proof of Theorem \ref{t2-2} and Theorem \ref{t2-3}}\label{section3}
\begin{proof}[Proof of Theorem \ref{t2-2}]

By Definition \ref{d2-1}, in order to verify that $(Y,Z)$ is a solution of  $N$-valued BSDE \eqref{d2-1-1}, it remains
to prove that $Y_t\in N$ a.s. for every $t\in [0,T]$ and $Z_t^i\in T_{Y_t}N$ for $dt\times \P$-a.e. $(t,\omega)\in [0,T]\times \Omega$ and
every $1\le i\le m$.

Let $\delta_0$ be the positive constant introduced in subsection \ref{section2-1} such that
the nearest projection map $P_N:B(N,3\delta_0)\to N$ and square of distance function
${\rm dist}^2_N: B(N,3\delta_0)\to \R_+$ are smooth.  Choosing a truncation function
$\chi\in C_b^\infty(\R)$ satisfying that $\chi'\ge 0$ and
\begin{equation*}
\chi(s)=
\begin{cases}
& s,\ \ \ \ \ \ \ s\le \delta_0^2,\\
& 4\delta_0^2,\ \ \ \ s>4\delta_0^2.
\end{cases}
\end{equation*}
We define $G:\R^L \to \R_+$ by 
\begin{equation}\label{e3-0}
G(p):=\chi\Big({\rm dist}^2_N(p)\Big),\ p\in \R^L.
\end{equation}
By the choice of $\delta_0$ and $\chi$ we have $G(p)=4\delta_0^2$ for every $p\in \R^L$ with
${\rm dist}_N(p)>2\delta_0$.
Note that $G(p)={\rm dist}_N^2(p)=|p-P_N(p)|^2$ when $p\in B(N,\delta_0)$,
it holds that for every $p\in B(N,\delta_0)$, $u=(u_1,\cdots, u_L)\in \R^L$,
\begin{equation*}
\begin{split}
\bar \nabla^2 G(p)(u,u)=&2\sum_{k=1}^L\left(\sum_{i=1}^L u_i\left(\delta_{ik}-\frac{\partial P_N^k}{\partial p_i}(p)\right)\right)^2 -
2\sum_{i,j,k=1}^L(p_k-P_N^k(p))\frac{\partial^2 P_N^k}{\partial p_i \partial p_j}(p)u_iu_j\\
&\ge -2\sum_{i,j,k=1}^L (p_k-P_N^k(p))\frac{\partial^2 P_N^k}{\partial p_i \partial p_j}(p)u_iu_j,
\end{split}
\end{equation*}
where $\delta_{ij}$ denotes the Kronecker delta function
(i.e. $\delta_{ij}=0$ if $i\neq j$ and $\delta_{ij}=1$ when $i=j$), $P_N^k(p)$ means the $k$-th components of $P_N(p)$, thus
$P_N(p)=\left(P_N^1(p),\cdots P_N^L(p)\right)$.
According to definition of $\bar A$ in  \eqref{e2-3} we have
for every $p\in B(N,\delta_0)$, $u=(u_1,\cdots, u_L)\in \R^L$,
\begin{equation*}
\begin{split}
\left\langle \bar \nabla G(p), \bar A(p)(u,u)\right\rangle&=
2\sum_{i,j,k=1}^L (p_k-P_N^k(p))\frac{\partial^2 P_N^k}{\partial p_i \partial p_j}(P_N(p))u_iu_j\\
&-
2\sum_{i,j,k,l=1}^L (p_k-P_N^k(p))\frac{\partial P_N^k}{\partial p_l}(p)\frac{\partial^2 P_N^l}{\partial p_i \partial p_j}(P_N(p))u_iu_j\\
&=2\sum_{i,j,k=1}^L (p_k-P_N^k(p))\frac{\partial^2 P_N^k}{\partial p_i \partial p_j}(P_N(p))u_iu_j.
\end{split}
\end{equation*}
Here in the last step above we have used the following equality
$$\sum_{k=1}^L(p_k-P_N^k(p))\frac{\partial P_N^k}{\partial p_l}(p)=0,$$
which is due to the property $\frac{\partial P_N}{\partial p_l}(p)\in T_p N$ and $p-P_N(p)\in T_p^{\bot} N$.
Combing above estimates together yields that
\begin{equation*}
\begin{split}
&\quad\ \bar \nabla^2 G(p)(u,u)+\left\langle \bar \nabla G(p), \bar A(p)(u,u)\right\rangle\\
&\ge  2\sum_{i,j,k=1}^L (p_k-P_N^k(p))\left(\frac{\partial^2 P_N^k}{\partial p_i \partial p_j}(P_N(p))
-\frac{\partial^2 P_N^k}{\partial p_i \partial p_j}(p)\right)u_iu_j\\
&\ge -c_2{\rm dist}_N^2(p)|u|^2=-c_2G(p)|u|^2,\ p\in B(N,\delta_0),\ u=(u_1,\cdots, u_L)\in \R^L.
\end{split}
\end{equation*}
Meanwhile for every $p \in B(N,\delta_0)$ and $u=(u_1,\cdots, u_L)\in \R^L$ we have
\begin{equation*}
\begin{split}
&\quad \left\langle \bar \nabla G(p), \bar f(p,u)\right\rangle\\
&=
2\sum_{k=1}^L (p_k-P_N^k(p))\bar f^k(p,u)-2\sum_{k,l=1}^L (p_k-P_N^k(p))\frac{\partial P_N^k(p)}{\partial p_l}(p)
\bar f^l(p,u)=0,
\end{split}
\end{equation*}
where the last step is due to the fact that $\frac{\partial P_N}{\partial p_l}(p)\in T_p N$, $\bar f(p,u)\in T_p N$ and
$p-P_N(p)\in T_p^{\bot} N$.

By all these estimates we arrive at
\begin{equation*}
\bar \nabla^2 G(p)(u,u)+\left\langle \bar \nabla G(p), \bar A(p)(u,u)-\bar f(p,u)\right\rangle
\ge -c_2G(p)|u|^2,\ p\in B(N,\delta_0),\ u\in \R^L.
\end{equation*}
Still by the definition of $G$, $\bar A$ and $\bar f$ we know that for every $p\in \R^L/B(N,\delta_0)$ and $u\in \R^L$,
\begin{align*}
&\bar \nabla^2 G(p)(u,u)+\left\langle \bar \nabla G(p), \bar A(p)(u,u)-\bar f(p,u)\right\rangle
\ge -c_3(1+|u|^2)\ge -c_4G(p)(1+|u|^2),
\end{align*}
where in the second inequality above we have used the fact that $G(p)\ge \delta_0^2$ for every $p\in \R^L/B(N,\delta_0)$.

Combing above  two estimates yields that
\begin{equation}\label{t2-2-4}
\bar \nabla^2 G(p)(u,u)+\left\langle \bar \nabla G(p), \bar A(p)(u,u)-\bar f(p,u)\right\rangle
\ge -c_5G(p)(1+|u|^2),\ \forall\ p,u\in \R^L.
\end{equation}
Hence by \eqref{d2-1-1}, \eqref{t2-2-4} and applying It\^o's formula we get
for every $t\in [0,T]$,
\begin{align*}
0=G(\xi)&=G(Y_t)+\sum_{i=1}^m\int_t^{T}\langle \bar \nabla G(Y_s), Z_s^i \rangle dB_s^i\\
&+
\sum_{i=1}^m\int_t^{T}\frac{1}{2}\Big(\bar \nabla^2 G(Y_s)(Z_s^i,Z_s^i)
+\left\langle \bar \nabla G(Y_s), \bar A(Y_s)(Z_s^i,Z_s^i)-2\bar f(Y_s,Z_s)\right\rangle\Big)ds\\
&\ge G(Y_t)+ \sum_{i=1}^m\int_t^{T}\langle \bar \nabla G(Y_s), Z_s^i \rangle dB_s^i-
\frac{c_5}{2}\int_t^{T}G(Y_s)(1+|Z_s|^2)ds\\
&\ge G(Y_t)+ \sum_{i=1}^m\int_t^{T}\langle \bar \nabla G(Y_s), Z_s^i \rangle dB_s^i
-c_6\int_t^{T}G(Y_s)ds
\end{align*}
Here we have applied \eqref{t2-2-2} and the fact that $G(\xi)=0$ a.s. (since $\xi \in N$ a.s.).
Taking the expectation in above inequality we arrive at
\begin{align*}
\E[G(Y_t)]\le c_6\int_t^{T}\E[G(Y_s)]ds,\ \forall\ t\in [0,T].
\end{align*}
So by Grownwall's inequality we obtain $\E[G(Y_t)]=0$ which implies $G(Y_t)=0$ and $Y_t\in N$ a.s. for
every $t\in [0,T]$.

As explained in the proof of \cite[Theorem 3.1]{KLT} (which is due to the original idea in \cite{PP2}), it holds that $Y_t\in \mathscr{D}^{1,2}(N)$ and we can find an equivalent version of $Z_t^i$ and $\D Y_t(\omega)(t)$ such that
\begin{align*}
Z_t^i(\omega)=\D Y_t(\omega)(t)\cdot e_i,\ \ dt\times \P-{\rm a.e.}\ (t,\omega)\in [0,T]\times \Omega,\ 1\le i \le m,
\end{align*}
where $e_i=(0,\cdots, \underbrace{1}_{i\
{\rm th}}, \cdots, 0)$, $1\le i \le m$
is the standard orthonormal basis of $\R^m$.

So according to \cite[Theorem 3.1]{Su} (concerning about characterization of $\mathscr{D}^{1,2}(\R^L)$), we know that
$Y_t$ is $\sigma(B_{\cdot})$ measurable and for every $t,r\in [0,T]$
\begin{align*}
\int_0^r \big(\D Y_t(s)\cdot e_i\big)ds=(\P)\lim_{\e \to 0}\frac{Y_t(B_{\cdot}+\e e_i^r(\cdot))-Y_t(B_{\cdot})}{\e},\ a.s.,
\end{align*}
where $(\P)\lim_{\e \to 0}$ denotes limit under the convergence in probability and
$e_i^r(t):=
(t\wedge r)e_i$. Based on this and the property that
$Y_t\in N$ a.s. we deduce that for every $t,r\in [0,T]$,
\begin{align*}
\int_0^r \big(\D Y_t (s)\cdot e_i\big)ds\in T_{Y_t}N,\ a.s..
\end{align*}
Therefore we can find a version of $Z_t^i$ such that
\begin{align*}
Z_t^i(\omega)=\D Y_t(\omega)(t)\cdot e_i\in T_{Y_t}N,\ \ dt\times \P-{\rm a.e.}\ (t,\omega)\in [0,T]\times \Omega,\ 1\le i \le m.
\end{align*}
Now we have proved the desired conclusion.
\end{proof}

\begin{proof} [Proof of Theorem \ref{t2-3}]
Now we assume that $m=1$.
In this proof we use the notation $\partial_x$, $\partial_{xx}^2$ to represent the first order and
second order derivative with respect
to $x\in \T^1$ respectively.

According to standard theory of quasi-linear parabolic equation (see e.g. \cite[Appendix A]{M} or \cite[Chapter V and VII]{LSU}),
there exists a $v\in C^1([0,T_1)\times \T^1;\R^L)\cap C^2((0,T_1)\times \T^1;\R^L)$ for some
(maximal time) $T_1>0$ which satisfies the following equation,
\begin{equation}\label{t2-2-4a}
\begin{cases}
& \partial_t v(t,x)-\frac{1}{2}\partial_{xx}^2 v(t,x)=-\frac{1}{2}\bar A\left(v(t,x)\right)\left(\partial_x v(t,x),\partial_x v(t,x)\right)
+\bar f\left(v(t,x),\partial_x v(t,x)\right),\\
& v(0,x)=h(x),\ \ \ \ t\in (0,T_1).
\end{cases}
\end{equation}
By the same arguments in the proof of Theorem \ref{t2-2} we will deduce that $v(t,\cdot)\in N$ for every $t\in [0,T_1)$. So we can replace
the terms $\bar A$, $\bar f$ by $A$ and $f$ in \eqref{t2-2-4a} respectively.
At the same time, by \eqref{t2-2-4a} we have for every $t\in (0,T_1)$,
\begin{align*}
\partial_t |\partial_x v|^2&=2\left\langle \partial_x\partial_t v, \partial_x v\right\rangle\\
&=2\left\langle \partial_x\left(\frac{1}{2}\partial^2_{xx} v-\frac{1}{2}A(v)\left(\partial_x v,\partial_x v\right)+f(v,\partial_x v)\right),\partial_x v\right\rangle\\
&=\left\langle \partial_{xxx}^3 v,\partial_x v\right\rangle-\left\langle \partial_x\left(A(v)
\left(\partial_x v,\partial_x v\right)\right),\partial_x v\right\rangle+
2\left\langle \partial_x\left(f(v,\partial_x v)\right),\partial_x v\right\rangle\\
&=:I_1+I_2+I_3.
\end{align*}
By direct computation we obtain
\begin{align*}
I_1&=\frac{1}{2}\partial_{xx}^2\left(|\partial_x v|^2\right)-|\partial_{xx}^2 v|^2.
\end{align*}
Since $\left\langle A(v)\left(\partial_x v,\partial_x v\right),\partial_x v\right\rangle=0$, we have
\begin{align*}
I_2&=-\partial_x \left(\left\langle A(v)\left(\partial_x v,\partial_x v\right),\partial_x v\right\rangle\right)+
\left\langle A(v)\left(\partial_x v,\partial_x v\right),\partial_{xx}^2 v\right\rangle\\
&=\left\langle A(v)\left(\partial_x v,\partial_x v\right),\partial_{xx}^2 v\right\rangle.
\end{align*}
Note that by \eqref{t2-2-4a} there is an orthogonal decomposition for $\partial_{xx}^2 v$ as follows
\begin{equation*}
\begin{split}
&\partial_{xx}^2 v=(\partial_{xx}^2 v)^T+(\partial_{xx}^2 v)^{\bot},\\
&(\partial_{xx}^2 v)^T:=2\partial_t v-2f(v,\partial_x v)\in T_v N,\\
&(\partial_{xx}^2 v)^{\bot}:=A(v)\left(\partial_x v,\partial_x v\right)\in T_v^{\bot} N.
\end{split}
\end{equation*}
So we obtain
\begin{align*}
I_1+I_2=\frac{1}{2}\partial_{xx}^2\left(|\partial_x v|^2\right)-|(\partial_{xx}^2v)^T|^2.
\end{align*}
By \eqref{e3-0a} we have
\begin{align*}
|I_3|&=2\left|\left\langle \nabla_{\partial_x v}\left(f(v,\partial_x v)\right),\partial_x v\right\rangle
\right|\\
&\le 2\left|\nabla_{\partial_x v}\left(f(v,\partial_x v)\right)\right||\partial_x v|\\
&\le 2\Big(\left|\nabla_p f(v,\partial_x v)\right||\partial_x v|+
\left|\nabla_u f(v,\partial_x v)\right||\nabla_{\partial_x v}\partial_x v |\Big)|\partial_x v|\\
&\le c_1\left(|\partial_x v|+|(\partial_{xx}^2 v)^{T}|\right)|\partial_x v|\le
|(\partial_{xx}^2 v)^{T}|^2+c_2|\partial_x v|^2.
\end{align*}
Here the fourth step above follows from the fact $\nabla_{\partial_x v}\partial_x v=
\big(\bar\nabla_{\partial_x v}\partial_x v\big)^T=(\partial_{xx}^2 v)^T$ and the last step is due to
Young's inequality.

Combing all above estimates together for $I_1$, $I_2$ and $I_3$ we arrive at
\begin{equation*}
\partial_t |\partial_x v|^2 \le \frac{1}{2}\partial_{xx}^2\left(|\partial_x v|^2\right)+c_2
|\partial_x v|^2,\ \forall\ t\in (0,T_1).
\end{equation*}
So for $e(t,x):=e^{-c_2t}|\partial_x v(t,x)|^2$ it holds,
\begin{equation*}
\partial_t e(t,x)\le \frac{1}{2}\partial_{xx}^2 e(t,x),\ \forall\ t\in (0,T_1).
\end{equation*}
Applying It\^o's formula to $e(t-s,B_s+x)$ directly we obtain
for every $\delta\in (0,T_1)$ and $t\in (\delta,T_1)$,
\begin{equation*}
\begin{split}
e(t,x)=e^{-c_2t}|\partial_x v(t,x)|^2
&\le \E[e(0,B_t+x)]=\int_{\T^1}\rho_{(0,x)}(t,y)|\partial_y h(y)|^2dy\\
&\le c_{3}\delta^{-1/2}\int_{\T^1}|\partial_y h(y)|^2dy,
\end{split}
\end{equation*}
where $\rho_{(0,x)}(t,y)$ is the heat kernel defined by \eqref{e3-3b} below.
This implies immediately that
\begin{equation}\label{t2-2-6}
\sup_{(t,x)\in [\delta,T_1)\times \T^1}|\partial_x v(t,x)|^2\le c_{3}e^{c_2T_1}
\delta^{-1/2}\int_{\T^1}|\partial_y h(y)|^2dy.
\end{equation}
So we have $\lim_{t \uparrow T_1}\sup_{x\in \T^1}|\partial_x v(t,x)|^2<\infty$,
hence by standard theory of quasi-linear parabolic equation, we could extend the solution
$v$ of \eqref{t2-2-4a} to time interval $(0,T_2]$ for some $T_2>T_1$. By the same arguments above
we can prove that \eqref{t2-2-6} holds with $T_1$ replaced by $T_2$. Therefore repeating this procedure again, we can extend
the solution $v$ of \eqref{t2-2-4a} to time interval $[0,T]$ for any $T>0$.

Then for any fixed $T>0$, suppose $v\in C^1([0,T]\times\T^1;N)\cap
C^2((0,T]\times \T^1;N)$ is the solution of \eqref{t2-2-4a} constructed above in time interval $[0,T]$.
We define $Y_t=v(T-t,B_t)$ and $Z_t:=\partial_x v(T-t,B_t)$ for $t\in [0,T]$, applying It\^o's formula directly we
can verify that $(Y,Z)$ is the unique  solution to $N$-valued BSDE \eqref{d2-1-1} which satisfies \eqref{t2-2-2} for some
$C_2>0$.

\end{proof}

\section{The proof of Theorem \ref{t2-1}}\label{section4}
In this section we will partly use the idea of \cite{CS,S2} (with some essential modification for the appearance of term
$\bar f $) to construct a solution to
$L^2(\T^m;N)$-valued BSDE \eqref{d2-2-2}.

Through this section, let $G:\R^L \to \R$ be defined by \eqref{e3-0} and we define $g:\R^L\to \R^L$ by
\begin{equation*}
\ g(p):=\bar \nabla G(p),\ \ \forall\ p\in \R^L.
\end{equation*}

For any $\e>0$, based on linear growth conditions \eqref{e3-1a} and  the fact $g\in C_b^\infty(\R^L;\R^L)$,
by standard theory of quasi-linear parabolic equation (see e.g. \cite[Chapter V and VII]{LSU},
or \cite[Appendix A]{M}),
there exists a unique solution $v_\e:[0,T]\times\T^m\to \R^L$
with $v_\e\in C^2((0,T]\times\T^m;\R^L)\cap
C^1([0,T]\times\T^m;\R^L)$ to following
equation
\begin{equation}\label{e3-1}
\begin{cases}
&\partial_t v_\e(t,x)-\frac{1}{2}\Delta_x v_\e(t,x)=-\frac{1}{2\e}g(v_\e(t,x))+\bar f(v_\e(t,x),\nabla_x v_\e(t,x)),\\
&v_\e(0,x)=h(x).
\end{cases}
\end{equation}

Inspired by \cite{CS,S2}, we are going to give several estimates for $v_\e$.
\begin{lem}\label{l3-1}
Suppose that $v_\e$ is the solution to \eqref{e3-1}, then for every $\e>0$, it holds that
\begin{equation}\label{l3-1-1}
\begin{split}
&\int_0^T \int_{\T^m} |\partial_t v_\e(t,x)|^2 dx dt+\sup_{t\in [0,T]}
\Big(\int_{\T^m}|\nabla_x v_\e(t,x)|^2 dx+\frac{1}{\e}\int_{\T^m} G(v_\e(t,x))dx\Big)\\
&\le e^{C_4T}\left(C_4T+\int_{\T^m}|\nabla_x h(x)|^2dx\right),
\end{split}
\end{equation}
where $C_4>0$ is a positive constant independent of $\e$ and $T$.
\end{lem}
\begin{proof}
We multiple both side of \eqref{e3-1} with $\partial_t v_\e$ to obtain
that for every $s\in [0,T]$,
\begin{align*}
\int_0^s \int_{\T^m}|\partial_t v_\e(t,x)|^2 dxdt=&\frac{1}{2}
\int_0^s \int_{\T^m} \langle \partial_t v_\e(t,x),\Delta_x v_\e(t,x)\rangle dxdt\\
&-\frac{1}{2\e}\int_0^s \int_{\T^m} \langle \bar \nabla G(v_\e(t,x)), \partial_t v_\e(t,x)\rangle dxdt\\
&+ \int_0^s \int_{\T^m}\langle \bar f\big(v_\e(t,x),\nabla v_\e(t,x)\big), \partial_t v_\e(t,x)\rangle dx dt\\
&=:I_1^\e+I_2^\e+I_3^\e.
\end{align*}
Since  $v_\e\in C^2((0,T]\times\T^m;\R^L)\cap
C^1([0,T]\times\T^m;\R^L)$, we obtain
\begin{align*}
I_1^\e &=-\frac{1}{2}\int_0^s \int_{\T^m}\langle \partial_t \nabla_x v_\e(t,x), \nabla_x v_\e(t,x)\rangle dx dt\\
&=-\frac{1}{4}\int_0^s \partial_t \Big(\int_{\T^m}|\nabla_x v_\e(t,x) |^2dx\Big)dt\\
&=\frac{1}{4}\int_{\T^m}|\nabla_x h(x)|^2dx-\frac{1}{4}\int_{\T^m}|\nabla_x v_\e(s,x)|^2dx.
\end{align*}
Note that $\langle\bar \nabla G(v_\e(t,x)), \partial_t v_\e(t,x)\rangle=\partial_t\big(G(v_\e(t,x))\big)$, it holds
\begin{align*}
I_2^\e&=-\frac{1}{2\e}\int_0^s \partial_t\Big(\int_{\T^m} G(v_\e(t,x))dx\Big)dt\\
&=-\frac{1}{2\e}\Big(\int_{\T^m}G(v_\e(s,x))dx-\int_{\T^m}G(h(x))dx\Big)=-
\frac{1}{2\e}\int_{\T^m}G(v_\e(s,x))dx.
\end{align*}
Here the last equality is due to the fact that $G(p)=0$ for every $p\in N$ and $h(x)\in N$ for a.e. $x\in \T^m$.
Meanwhile by \eqref{e3-1a} and Young inequality we have for every $s\in [0,T]$,
\begin{align*}
|I_3^\e|&\le \int_0^s \int_{\T^m}\Big(\frac{1}{2}|\partial_t v_\e(t,x)|^2+8|\bar f\big(v_\e(t,x),\nabla_x v_\e(t,x)\big)|^2\Big)dxdt \\
&\le \frac{1}{2}\int_0^s \int_{\T^m} |\partial_t v_\e(t,x)|^2dxdt +c_1\int_0^s\int_{\T^m}\big(1+|\nabla_x v_\e(t,x)|^2\big)dxdt\\
&\le \frac{1}{2}\int_0^s \int_{\T^m} |\partial_t v_\e(t,x)|^2dxdt+c_1
\int_0^s\int_{\T^m}|\nabla_x v_\e(t,x)|^2 dxdt+c_2T,
\end{align*}
where the positive constants $c_1,c_2$ are independent of $\e$. Therefore combing all above estimates
together yields that for every $s\in [0,T]$
\begin{align*}
&\int_0^s\int_{\T^m}|\partial_t v_\e(t,x)|^2dxdt+
\Big(\frac{1}{2}\int_{\T^m}|\nabla_x v_\e(s,x)|^2 dx+\frac{1}{\e}\int_{\T^m} G(v_\e(s,x))dx\Big) \\
&\le \frac{1}{2}\int_{\T^m}|\nabla_x h(x)|^2dx+2c_2T+2c_1\int_0^s \int_{\T^m} |\nabla_x v_\e(t,x)|^2 dxdt.
\end{align*}
Hence applying Grownwall lemma we can prove \eqref{l3-1-1}.

\end{proof}

Given a point $z_0=(t_0,x_0)\in [0,T]\times \T^m$, we define
\begin{equation}\label{e3-3a}
\begin{split}
& Q_R(z_0):=\{z=(t,x)\in [0,T]\times \T^m; x\in B_{\T^m}(x_0,R), |t-t_0|<R^2\},\ R\in (0,1/2),\\
& T_R(z_0):=\{z=(t,x)\in [0,T]\times \T^m; t_0-4R^2<t<t_0-R^2\},\ \ 0<R<\frac{\sqrt{t_0}}{2}.
\end{split}
\end{equation}
Also for any  $z_0=(t_0,x_0)\in [0,T]\times \T^m$, $0<R<\min(1/2,\sqrt{t_0}/2)$, let
\begin{equation}\label{e3-3b}
\rho_{z_0}(t,x):=\frac{1}{(2\pi|t_0-t|)^{m/2}}\exp\Big(-\frac{|x-x_0|^2}{2|t_0-t|)}\Big),\ t\in [0,T],\ x\in \T^m,
\end{equation}
\begin{equation}\label{e3-3c}
\begin{split}
& \Phi_\e(R):=R^2\int_{\T^m} \Big(\frac{1}{2}|\nabla_x v_\e(t_0-R^2/2,x)|^2+\frac{1}{\e}G(v_\e(t_0-R^2/2,x))\Big)
\rho_{z_0}(t_0-R^2/2,x)\varphi_{x_0}^2(x)dx,\\
&\ \ \ \ \ \ \ \ \ =R^2\int_{\R^m} \Big(\frac{1}{2}|\nabla_x v_\e(t_0-R^2/2,x)|^2+\frac{1}{\e}G(v_\e(t_0-R^2/2,x))\Big)
\rho_{z_0}(t_0-R^2/2,x)\varphi_{x_0}^2(x)dx,
\end{split}
\end{equation}
\begin{equation}\label{e3-3}
\begin{split}
& \Psi_\e(R):=\iint_{T_R(z_0)}\Big(\frac{1}{2}|\nabla_x v_\e(t,x)|^2+\frac{1}{\e}G(v_\e(t,x))\Big)
\rho_{z_0}(t,x)\varphi_{x_0}^2(x)dxdt\\
&\ \ \ \ \ \ \ \ =\int_{t_0-4R^2}^{t_0-R^2}\int_{\T^m}\Big(\frac{1}{2}|\nabla_x v_\e(t,x)|^2+\frac{1}{\e}G(v_\e(t,x))\Big)
\rho_{z_0}(t,x)\varphi_{x_0}^2(x)dxdt\\
&\ \ \ \ \ \ \ \ =\int_{t_0-4R^2}^{t_0-R^2}\int_{\R^m}\Big(\frac{1}{2}|\nabla_x v_\e(t,x)|^2+\frac{1}{\e}G(v_\e(t,x))\Big)
\rho_{z_0}(t,x)\varphi_{x_0}^2(x)dxdt.
\end{split}
\end{equation}
Here $\varphi_{x_0}\in C^\infty(\T^m;\R)$ is a cut-off function which satisfies that $\varphi_{x_0}(x)=1$ for every $x\in B_{\T^m}(x_0,1/4)$
, $\varphi_{x_0}(x)=0$ for every $x\in \T^m/B_{\T^m}(x_0,1/2)$ and
$\sup_{x_0\in \T^m}\|\varphi_{x_0}\|_\infty+\|\nabla_x \varphi_{x_0}\|_\infty<\infty$, and in the last equality of
\eqref{e3-3c} and \eqref{e3-3} we extend $\varphi_{x_0}$ to
a function defined on $\R^m$ with compact supports.
\begin{lem}\label{l3-2}
For any fixed $z_0=(t_0,x_0)\in [0,T]\times \T^m$, let $\Phi_\e(R)$, $\Psi_\e(R)$ be the functions
defined by \eqref{e3-3}, then for every $0<R\le R_0\le \min(1/2,\sqrt{t_0}/2)$,
\begin{equation}\label{l3-2-1}
\Phi_\e(R)\le e^{C_5(R_0-R)}\Phi_\e(R_0)+C_5(R_0-R),
\end{equation}
\begin{equation}\label{l3-2-2}
\Psi_\e(R)\le e^{C_5(R_0-R)}\Psi_\e(R_0)+C_5(R_0-R),
\end{equation}
where $C_5$ is a positive constant independent of $\e$ and $z_0=(t_0,x_0)$.
\end{lem}
\begin{proof}
In the proof, all the constants $c_i$ are independent of $\e$, $z_0$ and $R$. For every $1<t<4$ and $0<R\le R_0\le \min(1/2,\sqrt{t_0}/2)$, set
$v_\e^R(t,x):=v_\e(t_0-R^2 t,x_0+Rx)$. By \eqref{e3-1} we have immediately that
\begin{equation}\label{l3-2-3}
\partial_t v_\e^R(t,x)+\frac{1}{2}\Delta_x v_\e^R(t,x)=\frac{R^2}{2\e}g(v_\e^R(t,x))-\bar f^R(v_\e^R(t,x),\nabla_x v_\e^R(t,x)),
\end{equation}
where $\bar f^R:\R^L\times \R^{mL}\to \R^L$ is defined by $\bar f^R(p,u)=R^2\bar f(p,R^{-1}u)$.

Also note that $\rho_{z_0}(t_0-R^2 t,x_0+Rx)=R^{-m}\rho_{(0,0)}(t,x)$, applying integration by parts formula we obtain
\begin{align*}
\quad \Psi_\e(R)&=R^{2+m}\int_1^4 \int_{\R^m}
\Big(\frac{1}{2}|\nabla_x v_\e(t_0-R^2 t, x_0+Rx)|^2+\frac{1}{\e}G(v_\e(t_0-R^2 t, x_0+Rx))\Big)\\
&\ \ \times\rho_{z_0}(t_0-R^2 t,x_0+Rx)\varphi_{x_0}^2(x_0+Rx)dxdt\\
&=\int_1^4 \int_{\R^m}\frac{1}{2}|\nabla_x v_\e^R(t,x)|^2
\rho_{(0,0)}(t,x)\varphi_{x_0}^2(x_0+Rx)dxdt\\
&+\int_1^4 \int_{\R^m}\frac{R^2}{\e}G(v_\e^R(t, x))\rho_{(0,0)}(t,x)\varphi_{x_0}^2(x_0+Rx)dxdt\\
&=:I_1^{\e,R}+I_2^{\e,R}.
\end{align*}

Meanwhile according to integration by parts formula we have,
\begin{equation}\label{l3-2-4}
\begin{split}
\frac{\partial}{\partial R}I_1^{\e,R}
&=-\int_1^4 \int_{\R^m}\left\langle \Delta_x v_\e^R(t,x), \frac{\partial v_\e^R(t,x)}{\partial R}\right\rangle
\rho_{(0,0)}(t,x)\varphi_{x_0}^2(x_0+Rx)dxdt\\
&-\int_1^4 \int_{\R^m}\left\langle \nabla_x v_\e^R(t,x)\cdot \nabla_x\big(\rho_{(0,0)}(t,x)\varphi_{x_0}^2(x_0+Rx)\big),
\frac{\partial v_\e^R(t,x)}{\partial R}\right\rangle dxdt\\
&+\int_1^4 \int_{\R^m}|\nabla_x v_\e^R(t,x)|^2
\varphi_{x_0}(x_0+Rx)\rho_{(0,0)}(t,x)\big(\nabla_x\varphi_{x_0}(x_0+Rx)\cdot x\big) dxdt.
\end{split}
\end{equation}
Note that
\begin{align*}
\frac{\partial }{\partial R}v_\e^R(t,x)&=-2tR\partial_t v_\e(t_0-tR^2,x_0+Rx)+
\nabla_x v_\e(t_0-tR^2,x_0+Rx)\cdot x\\
&=\frac{1}{R}\Big(2t\partial_t v_\e^R(t,x)+\nabla_x v_\e^R(t,x)\cdot x\Big),
\end{align*}
and $\nabla_x \rho_{(0,0)}(t,x)=-\frac{x}{t}\rho_{(0,0)}(t,x)$, putting these estimates into \eqref{l3-2-4}
we arrive at
\begin{equation*}
\begin{split}
&\ \ \ \ \frac{\partial}{\partial R}I_1^{\e,R}\\
&=-\int_1^4\int_{\R^m}\frac{1}{R}\Big\langle  \Delta_x v_\e^R(t,x)-\frac{x}{t}\cdot\nabla_x v_\e^R(t,x),
2t\partial_t v_\e^R(t,x)+\nabla_x v_\e^R(t,x)\cdot x \Big\rangle \Theta_R(t,x) dxdt\\
&-2\int_1^4\int_{\R^m}\Big\langle \nabla_x v_\e^R(t,x)\cdot \nabla_x\varphi_{x_0}(x_0+Rx),
2t\partial_t v_\e^R(t,x)+\nabla_x v_\e^R(t,x)\cdot x \Big\rangle \Lambda_R(t,x)dxdt\\
&+\int_1^4 \int_{\R^m}|\nabla_x v_\e^R(t,x)|^2 \Lambda_R(t,x)
\big(\nabla_x\varphi_{x_0}(x_0+Rx)\cdot x\big) dxdt,
\end{split}
\end{equation*}
where $\Theta_R(t,x):=\rho_{(0,0)}(t,x)\varphi_{x_0}^2(x_0+Rx)$, $\Lambda_R(t,x):=\rho_{(0,0)}(t,x)\varphi_{x_0}(x_0+Rx)$.
By the same way we obtain
\begin{align*}
\frac{\partial}{\partial R}I_2^{\e,R}&=\int_1^4\int_{\R^m}\frac{1}{R}\Big\langle \frac{R^2}{\e}g\big(v_\e^R(t,x)\big),
2t\partial_t v_\e^R(t,x)+\nabla_x v_\e^R(t,x)\cdot x \Big\rangle \Theta_R(t,x) dxdt\\
&+2\int_1^4 \int_{\R^m}\frac{R^2}{\e}G\big(v_\e^R(t,x)\big)
\Lambda_R(t,x)\big(\nabla_x\varphi_{x_0}(x_0+Rx)\cdot x\big) dxdt\\
&+\int_1^4 \int_{\R^m}\frac{2R}{\e}G\big(v_\e^R(t,x)\big)\Theta_R(t,x) dxdt.
\end{align*}
Combing all above estimates for $\frac{\partial}{\partial R}I_1^{\e,R}$, $\frac{\partial}{\partial R}I_2^{\e,R}$
together and applying \eqref{l3-2-3} yields that
\begin{equation*}
\begin{split}
\frac{\partial}{\partial R}\Psi_\e(R)&=
\int_1^4\int_{\R^m}\frac{1}{tR}\Big|2t\partial_t v_\e^R(t,x)+\nabla_x v_\e^R(t,x)\cdot x \Big|^2\Theta_R(t,x)dxdt\\
&+\int_1^4 \int_{\R^m}\frac{2}{R}\Big\langle\bar f^R\big(v_\e^R(t,x),\nabla_x v_\e^R(t,x)\big),
2t\partial_t v_\e^R(t,x)+\nabla_x v_\e^R(t,x)\cdot x \Big\rangle \Theta_R(t,x)dxdt\\
&-2\int_1^4\int_{\R^m}\Big\langle \nabla_x v_\e^R(t,x)\cdot \nabla_x\varphi_{x_0}(x_0+Rx),
2t\partial_t v_\e^R(t,x)+\nabla_x v_\e^R(t,x)\cdot x \Big\rangle \Lambda_R(t,x)dxdt\\
&+\int_1^4 \int_{\R^m}\Big(|\nabla_x v_\e^R(t,x)|^2+\frac{2R^2}{\e}G\big(v_\e^R(t,x)\big)\Big)
\Lambda_R(t,x)\big(\nabla_x\varphi_{x_0}(x_0+Rx)\cdot x\big) dxdt\\
&+\int_1^4 \int_{\R^m}\frac{2R}{\e}G\big(v_\e^R(t,x)\big)\Theta_R(t,x) dxdt\\
&=:\sum_{i=1}^5J_{i}^{\e}(R)\ge \sum_{i=1}^4 J_{i}^{\e}(R) .
\end{split}
\end{equation*}
Since $|\bar f^R(p,u)|=R^2|\bar f(p,R^{-1}u)|\le c_1R(R+|u|)$, according to Young's inequality we obtain
\begin{align*}
|J_2^\e(R)|&\le \frac{1}{4}J_1^\e(R)+c_2\int_1^4\int_{\R^m}t(R^3+R|\nabla_x v_\e^R(t,x)|^2)\Theta_R(t,x)dxdt\\
&\le \frac{1}{4}J_1^\e(R)+c_3+c_4\int_1^4\int_{\R^m}R|\nabla_x v_\e^R(t,x)|^2\Theta_R(t,x)dxdt\\
&= \frac{1}{4}J_1^\e(R)+c_3+c_4\int_{t_0-4R^2}^{t_0-R^2}\int_{\R^m}R|\nabla_x v_\e(t,x)|^2\rho_{z_0}(t,x)\varphi_{x_0}^2(x)dxdt\\
&\le \frac{1}{4}J_1^\e(R)+c_4\Psi_\e(R)+c_3,
\end{align*}
where in the second inequality above we have applied the property
\begin{equation*}
\begin{split}
\int_1^4 \int_{\R^m} R^3 \Theta_R(t,x)dxdt&\le
R^3\|\varphi_{x_0}\|_\infty^2\int_1^4 \int_{\R^m}\rho_{(0,0)}(t,x)dxdt\le c_5R^3\le c_5\left(\frac{1}{2}\right)^3.
\end{split}
\end{equation*}

Still applying Young's inequality we get
\begin{align*}
&\quad |J_3^\e(R)|\\
&\le \frac{1}{4}J_1^\e(R)+c_6\int_1^4 \int_{\R^m}tR\Big|\nabla_x v_\e^R(t,x)\cdot \nabla_x\varphi_{x_0}(x_0+Rx)\Big|^2
\rho_{(0,0)}(t,x)dxdt\\
&= \frac{1}{4}J_1^\e(R)+c_6\int_1^4 \int_{B(x_0,1/2)/B(x_0,1/4)}tR^{3-m}\Big|\nabla_x v_\e(t_0-tR^2,x)\cdot \nabla_x\varphi_{x_0}(x)\Big|^2
\rho_{(0,0)}\big(t,\frac{x-x_0}{R}\big) dxdt\\
&\le \frac{1}{4}J_1^\e(R)+c_7\sup_{t\in [0,T]}\int_{\T^m}|\nabla_x v_\e(t,x)|^2 dx\le
\frac{1}{4}J_1^\e(R)+c_8.
\end{align*}
Here the second step from the change of variable and the fact that
$\nabla_x \varphi_{x_0}(x)\neq 0$ only if $x\in B(x_0,1/2)/B(x_0,1/4)$
(note that we still denote
the extension of $\varphi_{x_0}$ to a function on $\R^m$ with compact support by $\varphi_{x_0}$) , in the third step we
have applied the property that
\begin{equation*}
\sup_{R\in (0,1/2),t\in [1,4]}\sup_{x\in B(x_0,1/2)/B(x_0,1/4)}
R^{3-m}\rho_{(0,0)}\big(t,\frac{x-x_0}{R}\big)\le c_9\sup_{R\in (0,1/2)}R^{3-m}e^{-\frac{1}{128R^2}}\le c_{10},
\end{equation*}
and
\begin{equation*}
\int_{B(x_0,1/2)}|\nabla_x v_\e(t,x)|^2 dx \le \int_{\T^m}|\nabla_x v_\e(t,x)|^2 dx,
\end{equation*}
the last step is due to \eqref{l3-1-1}.

Handling $J_4^\e(R)$ by the same way of that for $J_3^\e(R)$ we arrive at
\begin{align*}
|J_4^\e(R)|&\le c_{11}\sup_{t\in [0,T]}
\Big(\int_{\T^m}|\nabla_x v_\e(t,x)|^2 dx+\frac{1}{\e}\int_{\T^m} G\big(v_\e(t,x)\big)dx\Big)\le c_{12}.
\end{align*}

Combing all above estimates for $J_i^\e(R)$, $i=1,2,3,4$ together yields that
\begin{align*}
\frac{\partial}{\partial R}\Psi_\e(R)&\ge \frac{1}{2}J_1^\e(R)-c_4\Psi_\e(R)-c_{13},\\
&\ge -c_4\Psi_\e(R)-c_{13}, \ \forall\ 0<R\le R_0.
\end{align*}
Applying Grownwall's lemma we obtain \eqref{l3-2-2} immediately.

The proof for \eqref{l3-2-1} is similar with that for \eqref{l3-2-2}, so we omit the details here.
\end{proof}

\begin{lem}\label{l3-3}
Given a $\e\in (0,1)$ and $R>0$, suppose that $v_{\e,R}\in C^2((0,T]\times \T^m;\R^L)$ satisfies the following equation
\begin{equation}\label{l3-3-0}
\partial_t v_{\e,R}(t,x)-\frac{1}{2}\Delta_x v_{\e,R}(t,x)=-\frac{R^2}{2\e}g(v_{\e,R}(t,x))+\bar f^R(v_{\e,R}(t,x),\nabla_x v_{\e,R}(t,x)),
\end{equation}
where $\bar f^R(p,u)=R^2 \bar f(p,R^{-1}u)$.
Set $e(v_{\e,R})(t,x):=\frac{1}{2}|\nabla_x v_{\e,R}(t,x)|^2+\frac{R^2}{\e}G\big(v_{\e,R}(t,x)\big)$.
Then there exists a positive constant $C_6>0$ such that for every $\e\in (0,1)$ and $R>0$,
\begin{equation}\label{l3-3-1}
\partial_t e(v_{\e,R})-\frac{1}{2}\Delta_x e(v_{\e,R})\le C_6e(v_{\e,R})\big(R^2+e(v_{\e,R})\big),\ \ \ \forall (t,x)\in (0,T]\times \T^m.
\end{equation}
\end{lem}
\begin{proof}
By \eqref{e2-2a} (see e.g. \cite[Section III.6]{C}) we know that
\begin{equation}\label{l3-3-2}
\bar \nabla {\rm dist}_N(p)\in T_{P_N(p)}^{\bot}N,\ \forall\ p\in B(N,3\delta_0),
\end{equation}
\begin{equation}\label{l3-3-3}
\Big|\bar \nabla \big({\rm dist}^2_N\big)(p)\Big|^2=4{\rm dist}^2_N(p),\ \ \ \forall\ p\in B(N,3\delta_0).
\end{equation}
Note that $G(u)=\chi\left({\rm dist}^2_N(p)\right)$, by direct computation we have
\begin{align*}
&\quad \frac{R^2}{\e}\left(\partial_t-\frac{1}{2}\Delta_x\right)G\left(v_{\e,R}\right)\\
&=\frac{R^2}{\e}\chi'\left({\rm dist}^2_N(v_{\e,R})\right)\left\langle \bar \nabla
 \left({\rm dist}^2_N\right)(v_{\e,R}), (\partial_t-\frac{1}{2}\Delta_x)v_{\e,R}\right\rangle\\
 &-\frac{R^2}{2\e}\left\langle \nabla_x \left(\chi'\left({\rm dist}^2_N(v_{\e,R})\right)\bar \nabla
 \left({\rm dist}^2_N\right)(v_{\e,R})\right)\cdot \nabla_x v_{\e,R}\right\rangle\\
 &=:I_1^{\e,R}+I_2^{\e,R}
\end{align*}
and
\begin{align*}
 \left(\partial_t-\frac{1}{2}\Delta_x\right)\frac{1}{2}\left|\nabla_x v_{\e,R} \right|^2
&=\left\langle \nabla_x \left(\partial_t v_{\e,R}-\frac{1}{2}\Delta_x v_{\e,R}\right)\cdot \nabla_x v_{\e,R}
\right\rangle-\frac{1}{2}|\nabla_x^2 v_{\e,R}|^2\\
&=:I_3^{\e,R}-\frac{1}{2}|\nabla_x^2 v_{\e,R}|^2.
\end{align*}
Here we use the notation to $\langle \cdot \rangle$ to denote the total inner product for
all the components in $\R^m$ and $\R^L$.
(For example,  $\left\langle \nabla_x \left(\partial_t v_{\e,R}-\frac{1}{2}\Delta_x v_{\e,R}\right)\cdot \nabla_x v_{\e,R}
\right\rangle$=$\sum_{i=1}^m\sum_{k=1}^L \partial_{x_i}\big(\partial_t v_{\e,R}^k$ $-\frac{1}{2}\Delta_x v_{\e,R}^k\big)$
$\partial_{x_i} v_{\e,R}^k$)

According to
\eqref{l3-3-3} and \eqref{l3-3-0} we find that
\begin{align*}
I_1^{\e,R}&=-\frac{R^4}{2\e^2}\left|\chi'\left({\rm dist}^2_N(v_{\e,R})\right)\right|^2
\left|\bar \nabla \left({\rm dist}^2_N\right)\left(v_{\e,R}\right)  \right|^2\\
&+\frac{R^2}{\e}
\chi'\left({\rm dist}^2_N(v_{\e,R})\right)\left\langle \bar f^R\left(v_{\e,R},\nabla_x v_{\e,R}\right),
\bar \nabla \left({\rm dist}^2_N\right)\left(v_{\e,R}\right) \right\rangle\\
&=-\frac{2R^4}{\e^2}\left|\chi'\left({\rm dist}^2_N\right)(v_{\e,R})\right|^2 {\rm dist}_N^2
\left(v_{\e,R}\right).
\end{align*}
Here in the last step we have used the property that
\begin{align*}
\left\langle \bar f^R(p,u), \bar \nabla {\rm dist}_N\left(p\right)\right\rangle=0,\
\ \forall\ p\in B(N,3\delta_0), u\in \R^{mL},
\end{align*}
which is due to the fact that $\bar f^R(p,u)\in T_{P_N(p)}N$ (see the definition \eqref{e2-5} of $\bar f$) and
$\bar \nabla \left({\rm dist}_N^2\right)\left(p\right)\in T_{P_N(p)}^{\bot}N$.

Note that for every $p\in B(N,3\delta_0)$, ${\rm dist}_N(p)^2=|p-P_N(p)|^2$,
hence for
every $p\in B(N,3\delta_0)$,
\begin{align*}
\frac{\partial^2 {\rm dist}_N^2}{\partial p_i\partial p_j}(p)=
2\sum_{k=1}^L\left(\left(\delta_{ik}-\frac{\partial P_N^k}{\partial p_i}(p)\right)
\left(\delta_{jk}-\frac{\partial P_N^k}{\partial p_j}(p)\right)-
\left(p_k-P_N^k(p)\right)\frac{\partial^2 P_N^k}{\partial p_i\partial p_j}(p)\right).
\end{align*}
Here $\delta_{ij}$ denotes the Kronecker delta function
(i.e. $\delta_{ij}=0$ if $i\neq j$ and $\delta_{ij}=1$ when $i=j$), $P_N^k(p)$ means the $k$-th components of $P_N(p)$, i.e.
$P_N(p)=\left(P_N^1(p),\cdots P_N^L(p)\right)$.

Based on this we obtain that when ${\rm dist}_N(v_\e)\le 2\delta_0$,
\begin{align*}
&\left\langle \nabla_x\left(\bar \nabla \left({\rm dist}_N^2\right)(v_{\e,R})\right)\cdot \nabla_x v_{\e,R}\right\rangle\\
&=\sum_{i,j=1}^L \sum_{l=1}^m\frac{\partial^2 {\rm dist}^2_N}{\partial p_i \partial p_j}\left(v_{\e,R}\right)\frac{\partial v_{\e,R}^i}{\partial x_l}
\frac{\partial v_{\e,R}^j}{\partial x_l}\\
&=2\sum_{k=1}^L\sum_{l=1}^m\left(\sum_{i=1}^L \frac{\partial v_{\e,R}^i}{\partial x_l}\left(\delta_{ik}-\frac{\partial P_N^k}{\partial p_i}(v_{\e,R})\right)\right)^2\\
&\quad -2\sum_{i,j,k=1}^L\sum_{l=1}^m\left(
\left(v_{\e,R}^k-P_N^k(v_{\e,R})\right)\frac{\partial^2 P_N^k}{\partial p_i\partial p_j}(v_{\e,R})\right)
\frac{\partial v_{\e,R}^i}{\partial x_l}
\frac{\partial v_{\e,R}^j}{\partial x_l}\\
&\ge -c_1{\rm dist}_N(v_{\e,R})|\nabla_x v_{\e,R}|^2,
\end{align*}
where in the last step we have used the fact $|v_{\e,R}^k-P_N^k(v_{\e,R})|\le {\rm dist}_N(v_{\e,R})$ and
$$ \sup_{p\in B(N,2\delta_0)}\Big|\frac{\partial^2 P_N^k}{\partial p_i\partial p_j}(p)\Big|\le c_2.$$

This along with the fact $\chi'\ge 0$ yields that when ${\rm dist}_N(v_{\e,R})\le 2\delta_0$,
\begin{align*}
\quad \ I_2^{\e,R}&\le \frac{c_1R^2}{2\e}\chi'\left({\rm dist}_N^2(v_{\e,R})\right){\rm dist}_N(v_{\e,R})|\nabla_x v_{\e,R}|^2
\\
&+\frac{R^2}{2\e}\left|\chi''\left({\rm dist}_N^2(v_{\e,R})\right)\right|\left|\bar \nabla ({\rm dist}_N^2)(v_{\e,R})\right|^2
|\nabla_x v_{\e,R}|^2\\
&\le \frac{R^4}{2\e^2} \left|\chi'\left({\rm dist}_N^2(v_{\e,R})\right)\right|^2
{\rm dist}_N^2(v_{\e,R})\\
&\quad +\frac{R^4}{\e^2}\left|\chi''\left({\rm dist}_N^2(v_{\e,R})\right)\right|^2{\rm dist}_N^4(v_{\e,R})
1_{\{{\rm dist}_N(v_{\e,R})\ge \delta_0\}}+c_3|\nabla_x v_{\e,R}|^4\\
&\le \frac{R^4}{2\e^2} \left|\chi'\left({\rm dist}_N^2(v_\e)\right)\right|^2
{\rm dist}_N^2(v_{\e,R})+ \frac{c_4R^4}{\e^2}G^2\left(v_{\e,R}\right)+c_3|\nabla_x v_{\e,R}|^4\\
&\le \frac{R^4}{2\e^2} \left|\chi'\left({\rm dist}_N^2(v_{\e,R})\right)\right|^2
{\rm dist}_N^2(v_{\e,R})+ c_5e(v_{\e,R})^2.
\end{align*}
Here second inequality follows from Young's inequality and the fact
$\chi''(s)\neq 0$ only when $s\ge \delta_0^2$,  in the third inequality we have applied
the property that
\begin{equation*}
\begin{split}
&\quad\left|\chi''\left({\rm dist}_N^2(v_{\e,R})\right)\right|^2 {\rm dist}_N^4(v_{\e,R})
1_{\{{\rm dist}_N(v_{\e,R})
\ge \delta_0\}}\\
&\le c_61_{\{{\rm dist}_N(v_{\e,R})\ge \delta_0\}}
\le \frac{c_6G^2\left(v_{\e,R}\right)}{\delta_0^4}1_{\{{\rm dist}_N(v_{\e,R})\ge \delta_0\}}\le
c_7G^2\left(v_{\e,R}\right).
\end{split}
\end{equation*}

By \eqref{l3-3-0} again we have
\begin{align*}
I_3^{\e,R}&=I_2^{\e,R}+\left\langle \nabla_x\left(\bar f^R\left(v_{\e,R},\nabla_x v_{\e,R}\right)\right)\cdot \nabla_x v_{\e,R} \right\rangle.
\end{align*}
According to \eqref{e3-1a} we obtain immediately that
\begin{align*}
\left\langle \nabla_x\left(\bar f^R\left(v_{\e,R},\nabla_x v_{\e,R}\right)\right)\cdot \nabla_x v_{\e,R} \right\rangle
&\le c_8\left(
R^2|\nabla_x v_{\e,R}|^2+R|\nabla_x v_{\e,R}|^3
+ R|\nabla^2_x v_{\e,R}| |\nabla_x v_{\e,R}|\right)\\
&\le \frac{1}{2}|\nabla^2_x v_{\e,R}|^2+c_9\left(R^2
|\nabla_x v_{\e,R}|^2+|\nabla_x v_{\e,R}|^4\right)\\
&\le \frac{1}{2}|\nabla^2_x v_{\e,R}|^2+c_{10}e(v_{\e,R})\left(R^2+e(v_{\e,R})\right),
\end{align*}
where in the second inequality above we have used Young's inequality.

Combing all above estimates for $I_1^{\e,R}$, $I_2^{\e,R}$, $I_3^{\e,R}$ together we can prove the desired conclusion
\eqref{l3-3-1}.
\end{proof}

\begin{rem}
Due to the appearance  of term $\bar f$, the solution $v_\e$ to \eqref{e3-1} is no longer scaling invariant. Therefore compared with
the method in \cite{CS} and \cite{S2}, in Lemma \ref{l3-2} and Lemma \ref{l3-3} above we could not only consider the situation
for $R=1$.
\end{rem}

\begin{lem}\label{l3-4}
Suppose that $\Psi_\e(R)$ is defined \eqref{e3-3}. There exist positive constants $\theta_0$ and $R_0\in (0,1/2)$ such that if for some
$(t_0,x_0)\in [0,T]\times \T^m$, $R<\min\{R_0,\sqrt{t_0}/2\}$, $\e\in (0,1)$,
\begin{equation}\label{l3-4-1}
\Psi_\e (R)=\iint_{T_R(z_0)}\Big(\frac{1}{2}|\nabla_x v_\e(t,x)|^2+\frac{1}{\e}G(v_\e(t,x))\Big)
\rho_{z_0}(t,x)\varphi_{x_0}^2(x)dxdt<\theta_0,
\end{equation}
then  we have
\begin{equation}\label{l3-4-2}
\sup_{(t,x)\in Q_{\kappa R}(z_0)}\left(|\nabla v_\e(t,x)|^2+\frac{1}{\e}G\left(v_\e(t,x)\right)\right)\le \frac{C_7}{\kappa^2 R^2}.
\end{equation}
Here $\kappa$  is a positive constant depending only on $E_0:=\int_{\T^m}|\nabla_x h(x)|^2dx$, $R$ (but independent of $\e$),
and $C_7$ is a positive constant independent of $\e$ and $R$.
\end{lem}
\begin{proof}
The proof is almost the same as that of \cite[Lemma 2.4]{CS} or \cite[Theorem 5.1]{S2}, the only difference here is
that we have to use the equation \eqref{e3-1} which is not scaling invariant. For convenience
of readers we also give the details here.

Set $e(v_\e):=\frac{1}{2}|\nabla_x v_\e|^2+\frac{1}{\e}G\left(v_\e\right)$. Let $r_1:=\kappa R$ for
some positive constant $\kappa\in (0,1/2)$ to be determined later.
In the proof we write $Q_r(z_0)$ for $Q_r$ with every $r>0$ for simplicity. Then we find an $r_0\in [0,r_1]$ such that
\begin{equation}\label{l3-4-3}
\sup_{0\le r \le r_1}\left\{(r_1-r)^2\sup_{(t,x)\in Q_r}e(v_\e)(t,x)\right\}=(r_1-r_0)^2\sup_{(t,x)\in Q_{r_0}}e(v_{\e})(t,x).
\end{equation}
Moreover, there exists a $z_1=(t_1,x_1)\in \overline{Q_{r_0}}$ such that
\begin{equation*}
\sup_{(t,x)\in Q_{r_0}}e(v_\e)(t,x)=e(v_\e)(t_1,x_1)=:e_0.
\end{equation*}
Let $s_0:=\frac{1}{2}(r_1-r_0)$, it is easy to see that $Q_{s_0}(z_1)\subset Q_{r_0+s_0}$. Hence by
\eqref{l3-4-3} we have
\begin{equation*}
\sup_{(t,x)\in Q_{s_0}(z_1)}e(v_\e)(t,x)\le \sup_{(t,x)\in Q_{r_0+s_0}}e(v_\e)(t,x)
\le \frac{(r_1-r_0)^2}{s_0^2}\sup_{(t,x)\in Q_{r_0}}e(v_\e)(t,x)= 4e_0.
\end{equation*}
Now set
\begin{equation*}
\begin{split}
&K_0:=\sqrt{e_0}s_0,\ v_{\e,e_0}(t,x):=v_\e\left(\frac{t}{e_0}+t_1,\frac{x}{\sqrt{e_0}}+x_1\right),\\
&e\left(v_{\e,e_0}\right)(t,x):=\frac{1}{2}|\nabla_x v_{\e,e_0}(t,x)|^2+\frac{1}{\e^2 e_0}G\left(v_{\e,e_0}(t,x)\right).
\end{split}
\end{equation*}
Obviously we have
\begin{equation}\label{l3-4-4}
\begin{split}
&e\left(v_{\e,e_0}\right)(0,0)=\frac{1}{e_0}e(v_\e)(t_1,x_1)=1,\\
&\sup_{(t,x)\in Q_{K_0}((0,0))}e\left(v_{\e,e_0}\right)(t,x)=
\frac{1}{e_0}\sup_{t,x\in Q_{s_0}(z_1)}e(v_\e)(t,x)\le 4.
\end{split}
\end{equation}
Meanwhile, it is not difficult to verify that $v_{\e,e_0}$ satisfies \eqref{l3-3-0} with $\e=\e$ and $R=\frac{1}{\sqrt{e_0}}$,
therefore according to Lemma \ref{l3-3} we have
\begin{equation}\label{l3-4-5}
\left(\partial_t-\frac{1}{2}\Delta_x\right) e\left(v_{\e,e_0}\right)\le c_1e\left(v_{\e,e_0}\right)\left(e_0^{-1}+e\left(v_{\e,e_0}\right)\right).
\end{equation}
Now we claim that $K_0:=\sqrt{e_0}s_0\le 1$. In fact, if $K_0>1$, then $e_0^{-1}< s_0^2\le T$, thus
by \eqref{l3-4-4} and \eqref{l3-4-5} it holds
\begin{equation*}
\begin{split}
\left(\partial_t-\frac{1}{2}\Delta_x\right) e\left(v_{\e,e_0}\right)\le c_1 e\left(v_{\e,e_0}\right)\left(T+e\left(v_{\e,e_0}\right)\right)
\le c_2e\left(v_{\e,e_0}\right)\ {\rm on}\ \ Q_{K_0}((0,0)).
\end{split}
\end{equation*}
Therefore for $\tilde e\left(v_{\e,e_0}\right):=e^{-c_2 t}e\left(v_{\e,e_0}\right)$ we have
\begin{equation*}
\left(\partial_t-\frac{1}{2}\Delta_x\right) \tilde e\left(v_{\e,e_0}\right)\le 0 \ {\rm on}\ \ Q_{K_0}((0,0)).
\end{equation*}
Since we assume that $K_0>1$, according to mean value theorem for sub-parabolic function
in \cite[Theorem 3]{Mo} or \cite[Theorem 5.2.9]{Sa} we obtain
\begin{equation}\label{l3-4-6}
\begin{split}
1=\tilde e\left(v_{\e,e_0}\right)(0,0)&\le c_3\int_{Q_1((0,0))}\tilde e\left(v_{\e,e_0}\right)(t,x)dtdx\\
&\le c_4 \int_{Q_1((0,0))}e\left(v_{\e,e_0}\right)(t,x)dtdx=c_4e_0^{\frac{m}{2}}
\int_{Q_{\frac{1}{\sqrt{e_0}}}(z_1)}e\left(v_{\e}\right)(t,x)dtdx.
\end{split}
\end{equation}
According to \eqref{l3-2-1}, \eqref{l3-2-2} and following the same arguments in the proof
of (2.19) in \cite[Lemma 2.4]{CS} (and also the comments in the proof of \cite[Lemma 4.4]{CS}), for any
$\delta>0$ we can find
$\kappa(\delta)\in (0,1)$ which may depend on $R$
and $c_5(\delta)>0$ independent of $R$ such that
for every $z\in Q_{r}$, $s>0$ with $r+s\le \kappa R$,
\begin{equation*}
\begin{split}
s^{-m}\int_{Q_s(z)}e\left(v_{\e}\right)(t,x)dtdx&\le
c_5\left(\Psi_\e(R)+RE_0\right)+\delta E_0\\
&\le c_5\left(\theta_0+R_0E_0\right)+\delta E_0,
\end{split}
\end{equation*}
where the last inequality follows from \eqref{l3-4-1} and the fact $R\le R_0$.

Note that $z_1\in Q_{r_0}$ and $\frac{1}{\sqrt{e_0}}+r_0< s_0+r_0\le r_1=\kappa R$, so for every $\delta>0$, we can find a
$c_6(\delta)>0$ such that
\begin{equation*}
\begin{split}
c_4e_0^{\frac{m}{2}}
\int_{Q_{\frac{1}{\sqrt{e_0}}}(z_1)}e\left(v_{\e,e_0}\right)(t,x)dtdx
&\le c_6\left(\theta_0+R_0E_0\right)+\delta E_0.
\end{split}
\end{equation*}
 Hence choosing
$\delta=\min\{\frac{1}{4E_0},\frac{1}{2}\}$, $\theta_0=\frac{1}{4c_6(\delta)}$,
$R_0=\min\{\frac{1}{4c_6E_0},\frac{1}{2}\}$ we get
\begin{equation*}
c_4e_0^{\frac{m}{2}}
\int_{Q_{\frac{1}{\sqrt{e_0}}}(z_1)}e\left(v_{\e,e_0}\right)(t,x)dtdx\le \frac{3}{4},
\end{equation*}
which is a contradiction to \eqref{l3-4-6}. So we  obtain that $K_0\le 1$. This along with
\eqref{l3-4-3} yields that for any $r\in [0,r_1]$,
\begin{equation*}
\begin{split}
(r_1-r)^2\sup_{(t,x)\in Q_r}e\left(v_\e\right)(t,x)&\le
\sup_{0\le r \le r_1}\left\{(r_1-r)^2\sup_{(t,x)\in Q_r}e\left(v_\e\right)(t,x)\right\}\\
&=(r_1-r_0)^2\sup_{(t,x)\in Q_{r_0}}e\left(v_\e\right)(t,x)=4s_0^2e_0=4K_0\le 4.
\end{split}
\end{equation*}
Therefore taking $r=\frac{r_1}{2}=\frac{\kappa R}{2}$ we can prove desired conclusion
\eqref{l3-4-2}.
\end{proof}

\begin{lem}\label{l3-5}
Let $R_0$, $\theta_0$ be the same constants in Lemma \ref{l3-4}, we define
\begin{equation}\label{l3-5-1}
\begin{split}
\Sigma:=\bigcap_{R\in (0,R_0)}&\Big\{z_0=(t_0,x_0)\in [0,T]\times \T^m;\\
& \liminf_{\e \to 0}
\iint_{T_R(z_0)}\Big(\frac{1}{2}|\nabla_x v_\e(t,x)|^2+\frac{1}{\e}G(v_\e(t,x))\Big)
\rho_{z_0}(t,x)\varphi_{x_0}^2(x)dxdt\ge \theta_0\Big\}.
\end{split}
\end{equation}
The $\Sigma$ is a closed subset of $[0,T]\times \T^m$ which has locally finite $m$-dimensional
Hausdorff measure with respect to the parabolic metric $\tilde d$ defined
by $\tilde d(z_1,z_2):=|t_1-t_2|^2+|x_1-x_2|$, $\forall z_1=(t_1,x_1)$, $z_2=(t_2,x_2)$.
\end{lem}
\begin{proof}
According to formula \eqref{l3-2-1}, \eqref{l3-2-2} (and the comments in the proof \cite[Lemma 4.4]{CS}), the proof
is exactly the same as that of \cite[Theorem 6.1]{S2}, so we do not include the details here.
\end{proof}

Now we start to prove Theorem \ref{t2-1}

\begin{proof}[Proof of Theorem \ref{t2-1}]
{\bf Step (i)} Suppose $v_\e$ is the solution to \eqref{e3-1}, set
\begin{equation*}
Y_t^{x,\e}:=v_\e(T-t,B_t+x), Z_t^{x,\e}:=\nabla_x v_\e(T-t,B_t+x),\ \ \forall\ (t,x)\in [0,T]\times \T^m.
\end{equation*}
Since $v_\e\in C^2((0,T]\times\T^m;\R^L)\cap C^1([0,T]\times \T^m;\R^L)$, applying It\^o's formula and \eqref{e3-1} we obtain immediately that
for every $(t,x)\in [0,T]\times \T^m$,
\begin{equation}\label{t3-1-1}
\begin{split}
Y_t^{x,\e}&=h\left(B_T+x\right)-\sum_{i=1}^m \int_t^T Z_s^{x,i,\e} dB_s^i-
\int_t^T \frac{1}{2\e}g\left(Y_s^{x,\e}\right)ds
+\int_t^T \bar f\left(Y_s^{x,\e},Z_s^{x,\e}\right)ds.
\end{split}
\end{equation}
According to the uniform estimates \eqref{l3-1-1} for $v_\e$, we can find a
function $v\in W^{1,2}([0,T];$ $ L^2(\T^m;\R^L))$ satisfying $\nabla_x v\in L^\infty([0,T];L^2(\T^m;\R^{mL}) )$ and a subsequence
$\{\e_k\}_{k=1}^\infty$ with $\lim_{k \to \infty}\e_k=0$, such that
\begin{equation}\label{t3-1-2}
\partial_t v_{\e_k}\to \partial_t v\ {\rm weakly\ in}\ L^2([0,T];L^2(\T^m;\R^L)),
\end{equation}
\begin{equation}\label{t3-1-3}
\nabla_x v_{\e_k}\to \nabla_x v\ {\rm weakly}^*\ {\rm in}\ L^\infty([0,T];L^2(\T^m;\R^{m L})).
\end{equation}
By \eqref{l3-1-1} again, Sobolev embedding theorem and diagonal principle, there exists a subsequence
$\{\e_k\}_{k=1}^\infty$ (through this proof we always denote it by $\{\e_k\}_{k=1}^\infty$ for simplicity) with $\lim_{k \to \infty}\e_k=0$
such that
\begin{equation}\label{t3-1-4}
\lim_{k \to \infty}\int_{\T^m}|v_{\e_k}(t,x)-v(t,x)|^2dxdt=0,\ \forall t\in \mathbb{Q}\cap [0,T],
\end{equation}
where $\mathbb{Q}$ denotes the collection of all the rational numbers as before. 
Still according to \eqref{l3-1-1} and the calculus for time involving Sobolev space (see
e.g. \cite[Theorem 2, Section 5.9.3]{EV}) we obtain for any $0\le s_1<s_2\le T$ and $\e\in (0,1)$,
\begin{equation}\label{t3-1-4a}
\begin{split}
\Big\|v_{\e}(s_1,\cdot)-v_{\e}(s_2,\cdot)\Big\|_{L^2(\T^m;\R^L)}&\le
\int_{s_1}^{s_2} \|\partial_t v_\e(t,\cdot)\|_{L^2(\T^m;\R^L)}dt\\
&\le \sqrt{\int_0^T\|\partial_t v_\e(t,\cdot)\|_{L^2(\T^m;\R^L)}^2dt}\sqrt{s_2-s_1}\\
&\le c_1\sqrt{s_2-s_1}.
\end{split}
\end{equation}
This along with \eqref{t3-1-4} yields that $v\in C([0,T];L^2(\T^m;\R^L))$, \eqref{t3-1-4a} holds for $v$ and
for every $t\in [0,T]$ we have (choosing a subsequence of $\{v_{\e_k}\}$ if necessary)
\begin{equation}\label{t3-1-5}
\lim_{k \to \infty}\int_{\T^m}\left|v_{\e_k}(t,x)-v(t,x)\right|^2dx=0.
\end{equation}
We define $Y_t^x:=v(T-t,B_t+x)$, $Z_t^x=\nabla_x v(T-t,B_t+x)$ for every $(t,x)\in [0,T]\times \T^m$.
So it follows from \eqref{t3-1-5} that
\begin{equation}\label{t3-1-6}
\lim_{k\to \infty}\E\left[\int_{\T^m}\left|Y_t^{x,\e_k}-Y_t^x\right|^2dx\right]=0,\ \forall\ t\in [0,T].
\end{equation}

For every $0\le s< t\le T$, it holds
\begin{align*}
&\int_{\T^m}|Y_t^x(\omega)-Y_{s}^x(\omega)|^2dx\\
&\le 2\int_{\T^m}|v(T-s,B_{s}(\omega)+x)-v(T-t,B_{s}(\omega)+x)|^2dx\\
&+2\int_{\T^m}|v(T-t,B_{s}(\omega)+x)-v(T-t,B_t(\omega)+x)|^2dx\\
&=2\int_{\T^m}|v(T-s,x)-v(T-t,x)|^2dx+2\int_{\T^m}|v(T-t,x+B_{s}(\omega)-B_t(\omega))-v(T-t,x)|^2dx\\
&=:I_1(s,t,\omega)+I_2(s,t,\omega).
\end{align*}
Applying the fact that \eqref{t3-1-4a} hold for $v$ we obtain
\begin{align*}
I_1(s,t,\omega)\le 2c_1^2|s-t|.
\end{align*}
Meanwhile by standard approximation procedure it is easy to verify
that for every fixed $t\in [0,T]$,
\begin{align*}
\lim_{y \to 0}\int_{\T^m}|v(t,x+y)-v(t,x)|^2dx=0.
\end{align*}
This, along with the continuity of $t\mapsto B_t(\omega)$, implies immediately that
we can find a null set $\Pi_0\subset \Omega$ such that
\begin{align*}
\lim_{s\to t}I_2(s,t,\omega)=0,\ \ \omega\notin \Pi_0,\ t\in [0,T].
\end{align*}
Combing all estimates above we deduce that $t\mapsto Y_t^{\cdot}(\omega)$ is continuous in $L^2(\T^m;\R^L)$ a.s..
According to this and the property that $\sup_{t\in [0,T]}\|v(t,\cdot)\|_{L^2(\T^m;\R^L)}<\infty$ we can prove
\begin{equation*}\label{t3-1-6a}
\E\left[\sup_{t\in [0,T]}\|Y_t^{\cdot}\|_{L^2(\T^m;\R^L)}^2\right]=
\sup_{t\in [0,T]}\|v(t,\cdot)\|_{L^2(\T^m;\R^L)}^2<\infty.
\end{equation*}

Hence by all the properties above we have verified that
$Y\in \S^2(\T^m;\R^L)$. At the same time, since $\nabla_x v\in L^\infty([0,T];L^2(\T^m;\R^L))$, we have immediately
that $Z\in \M^2(\T^m;\R^{L})$.

Moreover, for every $\psi\in C^2(\T^m;\R^L)$ it holds that
\begin{align*}
&\E\left[\int_0^T \left|\int_{\T^m}\left\langle Z_t^{x,i,\e_k},\psi(x)\right\rangle dx-
\int_{\T^m}\left\langle Z_t^{x,i},\psi(x)\right\rangle dx\right|^2 dt\right]\\
&=\E\left[\int_0^T \left|\int_{\T^m}\left\langle \frac{\partial v_{\e_k}}{\partial x_i}(T-t,B_t+x),\psi(x)\right\rangle dx-
\int_{\T^m}\left\langle \frac{\partial v}{\partial x_i}(T-t,B_t+x),\psi(x)\right\rangle dx\right|^2 dt\right]\\
&=\E\left[\int_0^T \left|\int_{\T^m}\left\langle v_{\e_k}(T-t,x),\frac{\partial \psi}{\partial x_i}(x-B_t)\right\rangle dx-
\int_{\T^m}\left\langle v(T-t,x),\frac{\partial \psi}{\partial x_i}(x-B_t)\right\rangle dx\right|^2 dt\right]\\
&\le c_2\|\nabla_x \psi\|_\infty^2 \int_0^T\int_{\T^m} |v_{\e_k}(t,x)-v(t,x)|^2dxdt.
\end{align*}
So by \eqref{t3-1-4a} and \eqref{t3-1-5} we have
\begin{align*}
\lim_{k \to \infty}\E\left[\int_0^T \left|\int_{\T^m}\left\langle Z_t^{x,i,\e_k},\psi(x)\right\rangle dx-
\int_{\T^m}\left\langle Z_t^{x,i},\psi(x)\right\rangle dx\right|^2 dt\right]=0,
\end{align*}
which implies that for each $t\in [0,T]$ and $1\le i \le m$,
\begin{equation}\label{t3-1-7}
\lim_{k \to \infty}\E\left[\Big|\int_t^T \left(\int_{\T^m}\langle Z_s^{x,i,\e_k}, \psi(x)\rangle dx \right)dB_s^i-
\int_t^T \left(\int_{\T^m}\langle Z_s^{x,i}, \psi(x)\rangle dx\right) dB_s^i\Big|^2\right]=0.
\end{equation}

{\bf Step (ii)} Let $\Sigma\subset [0,T]\times \T^m$ be defined by \eqref{l3-5-1}. By definition, for every
$z_0\notin \Sigma$, \eqref{l3-4-1} holds for some $R\in (0,R_0)$, therefore by \eqref{l3-4-2} we could find
a neighborhood $Q(z_0):=Q_{\kappa R}(z_0)$ of $z_0$ such that (taking a subsequence of $\{v_{\e_k}\}$ if necessary)
\begin{equation}\label{t3-1-8}
\sup_{k>0}\Big\{\sup_{(t,x)\in Q(z_0)}\Big(|\nabla_x v_{\e_k}(t,x)|^2+\frac{1}{\e_k}G\big(v_{\e_k}(t,x)\big)\Big)\Big\}\le c_3<\infty.
\end{equation}
Note that $G\big(v_{\e_k}(t,x)\big)=\chi\left({\rm dist}_N^2(v_{\e_k}(t,x))\right)$ and $\chi(s)\le \delta_0^2$ only if $s\le \delta_0^2$,
it follows from \eqref{t3-1-8} that for every $k$ large enough, 
\begin{equation*}
G\big(v_{\e_k}(t,x)\big)={\rm dist}_N^2\left(v_{\e_k}(t,x)\right)\le c_3\e_k,\ \forall\ (t,x)\in Q(z_0).
\end{equation*}
Note that we could find a countable collection of open neighborhoods $\{Q_i(z_0)\}_{i=1}^\infty$ as above to cover $[0,T]\times \T^m/\Sigma$,
by diagonal principle there exists a subsequence $\{v_{\e_k}\}$ such that
\begin{equation*}
\lim_{k \to \infty}G\big(v_{\e_k}(t,x)\big)=0,\ \ \forall\ (t,x)\in [0,T]\times \T^m/\Sigma.
\end{equation*}
By Lemma \ref{l3-5} we know that
$\Sigma$ has locally finite $m$-dimensional Hausdorff with respect to $\tilde d$, so under
the Lebesgue measure on $[0,T]\times \T^m$, $\Sigma$ is a null set. Therefore according to \eqref{t3-1-5}
it holds that
\begin{equation*}
G\big(v(t,x)\big)=0,\ \ dt\times dx-{\rm a.e.}\ (t,x)\in [0,T]\times \T^m,
\end{equation*}
which implies that $v(t,\cdot)\in L^2(\T^m;N)$ for a.e. $t\in [0,T]$. Combing this with
\eqref{t3-1-4a} we know that for every fixed $t\in [0,T]$.
\begin{equation*}
v(t,x)\in N,\ \ {\rm a.e.}\ x\in \T^m.
\end{equation*}
and $v(t,\cdot)\in L^2(\T^m;N)$.

By this we know for every fixed $t\in [0,T]$ and $1\le i \le m$,
\begin{equation*}
\partial_{x_i} v(t,x)\in T_{v(t,x)}N,\ \  {\rm a.e.}\ x\in \T^m.
\end{equation*}
Hence $Y_t^{\cdot}=v(T-t,B_t+\cdot)\in L^2(\T^m;N)$ for every $t\in [0,T]$ and
$Z_t^{i,x}=\partial_{x_i} v(T-t, B_t+x)\in T_{Y_t^x}N$ for $dt\times dx\times \P$-a.e. $(t,x,\omega)$. Note that it has been proved
that $Y\in \S^2(\T^m;\R^L)$, $Z\in \M^2(\T^m;\R^{L})$ 
in {\bf Step (i)} above, so we have
$(Y,Z)\in \S\otimes\M^2(\T^m;N)$.

{\bf Step (iii)}
Let $e(v_\e)(t,x):=\frac{1}{2}|\nabla_x v_\e(t,x)|^2+\frac{1}{\e}G(v_\e(t,x))$.
By the same methods (to estimate $I_1^{\e,R}$ and $I_2^{\e,R}$) in the proof of Lemma \ref{l3-3} we can prove
for every $\e\in (0,1)$,
\begin{equation*}
(\partial_t-\frac{1}{2}\Delta_x)e(v_\e)+\frac{1}{\e^2}\left|\chi'\left({\rm dist}_N^2(v_\e)\right)\right|^2
{\rm dist}_N^2(v_\e)\le c_4e(v_\e)(1+e(v_\e)),\ (t,x)\in (0,T]\times \T^m.
\end{equation*}
Combing this with \eqref{t3-1-8}, repeating the arguments in the proof of \cite[Theorem 3.1(Page 94)]{CS} and
the comments in the proof of \cite[Lemma 4.4]{CS} we can prove that for any open subset $Q'\subset Q(z_0)$ with $z_0\in
[0,T]\times \T^m/\Sigma$,
\begin{equation}\label{t3-1-9}
\sup_{k\ge 1}\iint_{Q'}|\nabla_x^2 v_{\e_k}(t,x)|^2 dtdx<\infty,
\end{equation}
and
\begin{equation}\label{t3-1-9aa}
\left(\partial_t-\frac{1}{2}\Delta_x\right)v_{\e_k} \to \left(\partial_t-\frac{1}{2}\Delta_x\right)v\ {\rm weakly\ in}\ L^2_{{\rm loc}}(Q(z_0)),
\end{equation}
\begin{equation}\label{t3-1-9a}
\frac{1}{\e_k} {\rm dist}_N(v_{\e_k}) \to \bar \lambda  \ {\rm weakly\ in}\ L^2_{{\rm loc}}(Q(z_0)),
\end{equation}
for some $\bar \lambda \in  L^2_{{\rm loc}}(Q(z_0))$. 

By \eqref{t3-1-9} we can find a subsequence $\{v_{\e_k}\}$  such that
\begin{equation}\label{t3-1-10}
\nabla_x^2 v_{\e_k}\to \nabla_x^2 v\  {\rm weakly\ in}\ L^2_{{\rm loc}}(Q(z_0)).
\end{equation}
At the same time, for every $\varphi\in C_c^\infty(Q(z_0))$ (here $C_c^\infty(Q(z_0))$
denotes the collection of smooth functions defined on $[0,T]\times \T^d$ whose supports are
contained in $Q(z_0)$), we have
\begin{equation*}
\begin{split}
\iint_{Q(z_0)} \left|\nabla_x v_{\e_k}(t,x)\right|^2 \varphi(t,x)dtdx
&=-\iint_{Q(z_0)} \varphi(t,x)\langle \Delta_x v_{\e_k}(t,x), v_{\e_k}(t,x)\rangle dtdx\\
&-\iint_{Q(z_0)}  \langle \nabla_x \varphi(t,x)\cdot \nabla_x v_{\e_k}(t,x), v_{\e_k}(t,x)\rangle dtdx.
\end{split}
\end{equation*}
Based on this expression, according to \eqref{t3-1-3}, \eqref{t3-1-4a}, \eqref{t3-1-5}, \eqref{t3-1-9} and \eqref{t3-1-10} we obtain
\begin{equation*}
\lim_{k \to \infty} \iint_{Q(z_0)} \left|\nabla_x v_{\e_k}(t,x)\right|^2 \varphi(t,x)dtdx=\iint_{Q(z_0)}
 \left|\nabla_x v(t,x)\right|^2 \varphi(t,x)dtdx.
\end{equation*}
This along with \eqref{t3-1-3} yields that for every $\varphi\in C_c^\infty(Q(z_0))$,
\begin{equation*}
\lim_{k \to \infty}\iint_{Q(z_0)} \left|\nabla_x v_{\e_k}(t,x)-\nabla_x v(t,x)\right|^2 \varphi(t,x)dtdx=0,
\end{equation*}
which means (take a subsequence if necessary)
\begin{equation*}
\lim_{k \to \infty}\nabla_x v_{\e_k}(t,x)=\nabla_x v(t,x),\ \ \ dt\times dx-{\rm a.e.}\ (t,x)\in Q(z_0).
\end{equation*}
Note that we could find a collection of countable open neighborhoods $\{Q_i(z_0)\}_{i=1}^\infty$ as above to cover $[0,T]\times \T^m/\Sigma$,
by diagonal principle there exists a subsequence $\{v_{\e_k}\}$ such that
(since the measure of $\Sigma$ is zero under $dt\times dx$)
\begin{equation*}\label{t3-1-11}
\lim_{k \to \infty}\nabla_x v_{\e_k}(t,x)=\nabla_x v(t,x),\ \ \ dt\times dx-{\rm a.e.}\ (t,x)\in [0,T]\times \T^m.
\end{equation*}
This together with \eqref{t3-1-5} implies immediately that (taking a subsequence if necessary)
for
\begin{equation}\label{t3-1-11a}
\lim_{k \to \infty}\bar f\left(v_{\e_k}(t,x),\nabla_x v_{\e_k}(t,x)\right)=
\bar f\left(v(t,x),\nabla_x v(t,x)\right),\ dt\times dx\ {\rm a.e.}-(t,x)\in [0,T]\times \T^m,
\end{equation}
Meanwhile by \eqref{e3-1a} and \eqref{l3-1-1} it is easy to verify that $\bar f\left(v_{\e_k}(T-t,x),\nabla_x v_{\e_k}(T-t,x)\right)$
is uniformly integrable with respect to $dt\times dx$ since
\begin{align*}
&\sup_{k\ge 1}\int_0^T \int_{\T^m} \left|\bar f\left(v_{\e_k}(T-t,x),\nabla_x v_{\e_k}(T-t,x)\right)\right|^2dtdx\\
&\le c_5\left(1+\sup_{k\ge 1}\int_0^T \int_{\T^m}|\nabla_x v_{\e_k}(T-t,x)|^2dtdx\right)<\infty.
\end{align*}
According to this and \eqref{t3-1-11a} we obtain that for every $\hat \psi\in L^\infty([0,T]\times \T^m;\R^L)$,
\begin{align*}
&\lim_{k \to \infty}
 \int_0^T \int_{\T^m}\left\langle \bar f\left(v_{\e_k}(t,x),\nabla_x v_{\e_k}(t,x)\right),
\hat \psi(t,x)\right\rangle dxdt\\
&=\int_0^T \int_{\T^m}\left\langle \bar f\left(v(t,x),\nabla_x v(t,x)\right),
\hat \psi(t,x)\right\rangle dxdt.
\end{align*}
Hence 
for every $\psi\in C^2(\T^m;\R^L)$, $t\in [0,T]$ and
a.s. $\omega\in \Omega$,
\begin{equation}\label{t3-1-12}
\begin{split}
&\quad \lim_{k \to \infty} \int_t^T \int_{\T^m}\left\langle \bar f\left(Y_s^{x,\e_k},Z_s^{x,\e_k}\right),\psi(x)\right\rangle dxds\\
&=\lim_{k \to \infty} \int_t^T \int_{\T^m}\left\langle \bar f\left(v_{\e_k}(T-s,B_s+x),\nabla_x v_{\e_k}(T-s,B_s+x)\right),\psi(x)\right\rangle dxds\\
&=\lim_{k \to \infty} \int_t^T \int_{\T^m}\left\langle \bar f\left(v_{\e_k}(T-s,x),\nabla_x v_{\e_k}(T-s,x)\right),
\psi(x-B_s)\right\rangle dxds\\
&=\int_t^T \int_{\T^m}\left\langle \bar f\left(v(T-s,x),\nabla_x v(T-s,x)\right),
\psi(x-B_s)\right\rangle dxds\\
&=\int_t^T \int_{\T^m}\left\langle \bar f\left(Y_s^{x},Z_s^{x}\right),\psi(x)\right\rangle dxds
\end{split}
\end{equation}

{\bf Step (iv)} Note that by \eqref{t3-1-8} we know $\chi'\left({\rm dist}_N^2(v_{\e_k})\right)=1$ on
$Q(z_0)$ when $k$ is large enough. So as explained in the proof of Lemma \ref{l3-3}, we have
\begin{align*}
\frac{1}{\e_k}g\left(v_{\e_k}\right)&=\frac{1}{\e_k}\bar \nabla {\rm dist}_N^2(v_{\e_k})\\
&=\frac{2}{\e_k}{\rm dist}_N(v_{\e_k})\bar \nabla {\rm dist}_N(v_{\e_k})\in T^{\bot}_{P_N\left(v_{\e_k}\right)}N
\end{align*}
Hence combing this with \eqref{t3-1-9aa}, \eqref{t3-1-9a}, \eqref{t3-1-11a} and following the same arguments in
the proof of \cite[Theorem 3.1(Page 94--95)]{CS} we obtain that for $dt\times dx$-a.e $(t,x)\in [0,T]\times \T^m/\Sigma$,
\begin{equation*}
\Big\{\left(\partial_t-\frac{1}{2}\Delta_x\right)v(t,x)-\bar f\left(v(t,x),\nabla_x v(t,x)\right)\Big\}\bot T_v N.
\end{equation*}
From this we deduce that for $dt\times dx$-a.e $(t,x)\in [0,T]\times \T^m/\Sigma$,
\begin{equation}\label{t3-1-13}
\begin{split}
\left(\partial_t-\frac{1}{2}\Delta_x\right)v-\bar f\left(v,\nabla_x v\right)&=
\sum_{j=1}^{L-n}\left\langle\left(\partial_t-\frac{1}{2}\Delta_x\right)v-\bar f\left(v,\nabla_x v\right), \nu_j(v)\right\rangle\nu_j(v),\\
&=\sum_{j=1}^{L-n}-\left\langle \frac{1}{2}\Delta_x v, \nu_j(v)\right\rangle\nu_j(v)\\
&=-\frac{1}{2}\sum_{i=1}^m A(v)\left(\partial_{x_i} v,\partial_{x_i} v\right), 
\end{split}
\end{equation}
where $\{\nu_i(p)\}_{j=1}^{L-n}$ is an orthonormal basis of $T_p^{\bot}N$ at $p\in N$,
in the second equality  above we have
used the fact $\bar f\left(v,\nabla_x v\right)\in T_v N$, $\partial_t v\in T_v N$ for a.e. $x\in \T^m$, and
the last step follows from the standard property of sub-manifold (see e.g. \cite[Section 1.3]{LW}).

Given \eqref{t3-1-13} and applying the same procedures in the proof of \cite[Theorem 3.1(Page 95)]{CS}
(using again the fact that the measure of $\Sigma$ is zero under $dt\times dx$) we obtain that
for every $\hat \psi\in L^\infty([0,T]\times\T^m;\R^L)$,
\begin{equation*}
\begin{split}
\int_0^T \int_{\T^m}
\left\langle \partial_t v, \hat \psi\right\rangle+\frac{1}{2}\sum_{i=1}^m\Big(\left\langle \partial_{x_i} v,\partial_{x_i}\hat \psi\right\rangle
+\left\langle A(v)(\partial_{x_i} v,\partial_{x_i} v),\hat \psi\right\rangle\Big)-
\left\langle \bar f\left(v,\nabla_x v\right),\hat \psi\right\rangle dtdx=0.
\end{split}
\end{equation*}
Combing this with the equation \eqref{e3-1} and the convergence property \eqref{t3-1-2},\eqref{t3-1-3},\eqref{t3-1-12} it
holds that
\begin{equation}\label{t3-1-14}
\lim_{k \to \infty}\int_0^T \int_{\T^m}\left\langle \frac{1}{\e_k}g\left(v_{\e_k}\right),\hat \psi\right\rangle dtdx=
\sum_{i=1}^m\int_0^T \int_{\T^m}\left\langle  A(v)(\partial_{x_i} v,\partial_{x_i} v),\hat \psi\right\rangle dtdx.
\end{equation}
For any $\psi\in C^2(\T^m;\R^L)$ and $t\in [0,T]$, taking $\hat \psi(s,x)=\psi(T-s,x-B_s)1_{[0,T-t]}(s)$ in \eqref{t3-1-14} where $\psi(s,x)\equiv \psi(x) \ \forall s\in [0,T]$,
we obtain that
for a.s. $\omega\in \Omega$ 
\begin{align*}
&\lim_{k \to \infty}\int_t^T\int_{\T^m} \left\langle \frac{1}{\e_k}g\left(Y_s^{x,\e_k}\right),\psi(s,x)\right\rangle dxds\\
&=\lim_{k \to \infty}\int_t^T \int_{\T^m}\left\langle \frac{1}{\e_k}g\left(v_{\e_k}(T-s,B_s+x)\right),\psi(s,x)\right\rangle dxds\\
&=\lim_{k \to \infty}\int_0^T \int_{\T^m}\left\langle \frac{1}{\e_k}g\left(v_{\e_k}(s,x)\right),\psi(T-s,x-B_s)1_{[0,T-t]}(s)\right\rangle dxds\\
&=\sum_{i=1}^m\int_0^T \int_{\T^m}\left\langle  A(v(s,x))\left(\partial_{x_i} v(s,x),\partial_{x_i} v(s,x)\right),\psi(T-s,x-B_s)1_{[0,T-t]}(s)\right\rangle dxds\\
&=\sum_{i=1}^m\int_t^T\int_{\T^m} \left\langle A(Y_s^{x})\left(Z_s^{x,i},Z_s^{x,i}\right),\psi(s,x)\right\rangle dxds.
\end{align*}

Putting this with \eqref{t3-1-6},\eqref{t3-1-7}, \eqref{t3-1-12} into \eqref{t3-1-1} we can verify that
for every $t\in [0,T]$ and $\psi\in C^2(\T^m;\R^L)$ , \eqref{r2-1-1} holds for a.s. $\omega\in \Omega$. By Proposition
 \ref{r2-1} we have finished the proof.

\end{proof}

\vskip 5mm

\noindent \textbf{Acknowledgements.}
The research of Xin Chen and Wenjie Ye is supported by the National Natural Science Foundation of China (No.\ 11871338)

\end{document}